\documentclass{article}

\usepackage{microtype}
\usepackage{graphicx}
\usepackage{subcaption}
\usepackage{booktabs} 

\usepackage{hyperref}

\usepackage[preprint]{icml2026}

\usepackage{amsmath}
\usepackage{amssymb}
\usepackage{mathtools}
\usepackage{amsthm}

\allowdisplaybreaks

\usepackage[capitalize,noabbrev]{cleveref}
\Crefname{assumption}{Assumption}{Assumptions}

\icmltitlerunning{Broximal Alignment for Global Non-Convex Optimization}

\usepackage{multirow}
\usepackage{layout}

\usepackage{amsmath}
\usepackage{amsthm}
\usepackage{amssymb}
\usepackage{amsfonts}
\usepackage{mathtools}
\usepackage{siunitx}

\usepackage{enumitem}

\usepackage{bbm}

\usepackage{url}
\usepackage{color}
\usepackage{graphicx}
\usepackage{verbatim}
\usepackage{nicefrac}

\usepackage[export]{adjustbox}

\usepackage{apptools}
\usepackage{thmtools}
\usepackage{thm-restate}
\usepackage{mdframed}

\usepackage[flushleft]{threeparttable}
\usepackage{array,booktabs,makecell}

\usepackage{listings}
\usepackage{xcolor}
\usepackage{todonotes}

\newcommand{\norm}[1]{\left\| #1 \right\|}

\newcommand{\inp}[2]{\left\langle#1,#2\right\rangle} 

\newcommand{\parens}[1]{\left( #1 \right)}
\newcommand{\brac}[1]{\left\{ #1 \right\}}

\newcommand{\cB}{\mathcal{B}}
\newcommand{\cC}{\mathcal{C}}

\newcommand{\cF}{\mathcal{F}}

\newcommand{\cN}{\mathcal{N}}

\newcommand{\cX}{\mathcal{X}}

\newcommand{\mQ}{\mathbf{Q}}
\newcommand{\mX}{\mathbf{X}}
\newcommand{\mI}{\mathbf{I}}

\newcommand{\del}[1]{}

\newcommand{\R}{\mathbb{R}} 

\newcommand{\eqdef}{:=} 

\DeclareMathOperator{\dom}{dom}         

\DeclareMathOperator*{\argmin}{arg\,min}

\def\flr#1{\left\lfloor #1 \right\rfloor}

\newcommand{\BProxSub}[3]{\textnormal{{brox}}^{#1}_{#2}(#3)}
\newcommand{\Proj}[2]{\Pi_{#1}(#2)}
\newcommand{\dist}[2]{\operatorname{dist}\!\left(#1,#2\right)}

\usepackage{tcolorbox}
\usepackage{pifont}
\definecolor{mydarkgreen}{RGB}{39,130,67}
\definecolor{mydarkred}{RGB}{192,47,25}
\definecolor{mydarkorange}{RGB}{39,130,67}
\definecolor{yaleblue}{rgb}{0.06, 0.3, 0.57}
\definecolor{myred}{RGB}{215,60,50}
\definecolor{coral}{HTML}{FF7F50}
\definecolor{peach}{HTML}{CC5500}
\definecolor{NavyBlue}{RGB}{0, 0, 128}
\definecolor{CrimsonRed}{RGB}{220, 20, 60}
\definecolor{Gold}{RGB}{204, 172, 0}
\definecolor{PaleRed}{rgb}{0.95, 0.3, 0.3}
\definecolor{MPLOrange}{rgb}{1.0, 0.6470588235294118, 0.0}
\definecolor{MPLGreen}{rgb}{0.0, 0.5019607843137255, 0.0}
\definecolor{ForestGreen}{RGB}{34,139,34}
\definecolor{OliveGreen}{RGB}{107,142,35}
\definecolor{PurplePrint}{RGB}{117, 112, 179}
\definecolor{GreenPrint}{RGB}{27, 158, 119}
\definecolor{RedPrint}{RGB}{217, 95, 2}
\definecolor{VeryLightGray}{rgb}{0.9,0.9,0.9}

\newcommand{\green}{\color{mydarkgreen}}
\newcommand{\red}{\color{mydarkred}}
\newcommand{\cmark}{\green\ding{51}}%
\newcommand{\xmark}{\red\ding{55}}%

\newcommand{\algnamesmall}[1]{{\small \sf #1}}

\newcommand{\newalg}{\algnamesmall{BPM}}

\newcommand{\newclass}[2]{\mathcal{F}_{#1}^{\textnormal{BA}}(#2)}
\newcommand{\newclassu}[2]{\mathcal{F}_{#1}^{\textnormal{UBA}}(#2)}

\mdfdefinestyle{theoremstyle}{
    backgroundcolor=lightgray!12, 
    linecolor=lightgray!12, 
    innertopmargin=6pt, 
    innerbottommargin=4pt,
    innerrightmargin=5pt,
    innerleftmargin=5pt
}

\declaretheoremstyle[
    headfont=\bfseries,
    bodyfont=\normalfont,
    mdframed={
        style=theoremstyle
    }
]{graytheoremstyle}

\declaretheorem[
    name=Theorem,
    style=graytheoremstyle,
    numberwithin=section
]{theorem}

\declaretheorem[
    name=Lemma,
    style=graytheoremstyle,
    numberwithin=section
]{lemma}

\declaretheorem[
    name=Assumption,
    style=graytheoremstyle,
    numberwithin=section
]{assumption}

\declaretheorem[
    name=Definition,
    style=graytheoremstyle,
    numberwithin=section
]{definition}

\newtheorem{proposition}{Proposition}

\newtheorem{fact}{Fact}

\theoremstyle{plain}

\newtheorem{example}{Example}
\newtheorem{remark}[theorem]{Remark}

\definecolor{codegreen}{rgb}{0,0.6,0}
\definecolor{codegray}{rgb}{0.5,0.5,0.5}
\definecolor{codepurple}{rgb}{0.58,0,0.82}
\definecolor{backcolour}{rgb}{0.95,0.95,0.92}

\lstdefinestyle{mystyle}{
    backgroundcolor=\color{backcolour},   
    commentstyle=\color{codegreen},
    keywordstyle=\color{magenta},
    numberstyle=\tiny\color{codegray},
    stringstyle=\color{codepurple},
    basicstyle=\ttfamily\footnotesize,
    breakatwhitespace=false,         
    breaklines=true,                 
    captionpos=b,                    
    keepspaces=true,                 
    numbers=left,                    
    numbersep=5pt,                  
    showspaces=false,                
    showstringspaces=false,
    showtabs=false,                  
    tabsize=2
}

\begin{document}

\twocolumn[
  \icmltitle{Broximal Alignment for Global Non-Convex Optimization}



  \icmlsetsymbol{equal}{*}

  \begin{icmlauthorlist}
    \icmlauthor{Kaja Gruntkowska}{yyy}
    \icmlauthor{Hanmin Li}{yyy}
    \icmlauthor{Xun Qian}{yyy}
    \icmlauthor{Peter Richt\'{a}rik}{yyy}
  \end{icmlauthorlist}

  \icmlaffiliation{yyy}{Center of Excellence for Generative AI, King Abdullah University of Science and Technology, Thuwal, Saudi Arabia}

  \icmlcorrespondingauthor{Kaja Gruntkowska}{kaja.gruntkowska@kaust.edu.sa}

  \icmlkeywords{Machine Learning, ICML}

  \vskip 0.3in
]



\printAffiliationsAndNotice{}  

\begin{abstract}
    Most non-convex optimization theory is built around gradient dynamics, leaving global convergence largely unexplored. The dominant paradigm focuses on stationarity, certifying only that the gradient norm vanishes, which is often a weak proxy for actual optimization success. In practice, gradient norms can stagnate or even increase during training, and stationary points may be far from global solutions.
    In this work, we propose a new framework for global non-convex optimization that avoids gradient-based reasoning altogether. We revisit the Ball Proximal Point Method ({\newalg}), a trust-region–style algorithm introduced by \citet{gruntkowska2025ball}, and propose a novel structural condition -- \emph{Broximal Alignment} -- under which {\newalg} provably converges to a global minimizer. Our condition requires no convexity, smoothness, or Lipschitz assumptions, and it permits multiple and disconnected global minima as well as non-optimal local minima.
    We show that this class generalizes standard non-convex frameworks such as quasiconvexity, star convexity, quasar convexity, and the aiming condition. Our results provide a new conceptual foundation for global non-convex optimization beyond stationarity.
\end{abstract}

\section{Introduction}

Non-convex optimization lies at the heart of modern machine learning (ML). Matrix factorization \citep{koren2009matrix}, phase retrieval \citep{patterson1934fourier, candes2015phase}, representation learning, and most prominently deep neural networks \citep{lecun2015deep, goodfellow2016deep} all lead to objective functions with highly non-convex landscapes. Although such functions provide the expressive capacity needed to model complex learning tasks, they often pose significant challenges for optimizer designers. The underlying landscape can be highly irregular, with many poor local minima and saddle points that can trap optimization algorithms \citep{pascanu2014saddle, vidal2017mathematics, safran2018spurious, jin2021nonconvex}.

Unlike convex optimization, where global minimizers can be efficiently found \citep{nemirovski1983problem, nemirovskii1985optimal, nesterov2003introductory}, we lack a comparable toolkit for non-convex problems. Indeed, many non-convex optimization tasks are NP-hard to solve \citep{murty1987someNP}, and a number of them are NP-hard even to approximate \citep{meka2008rank, jain2017non}. As a result, rather than aiming for (approximate) global solutions, most theoretical results for non-convex optimization, especially those for gradient-based methods, are framed in terms of \emph{stationarity}: convergence to points where the gradient norm is small. Although mathematically convenient, such guarantees provide limited insight into \emph{global} optimization behavior.

\subsection{Limitations of Stationarity-Based Non-convex Optimization Theory}

To fix notation, we consider solving the general (potentially non-convex) optimization problem
\begin{align}\label{eq:problem}
    \min_{x \in \R^d} f(x),
\end{align}
where $f: \R^d \mapsto \R \cup \{+\infty\}$ is proper (that is, the set $\dom f \eqdef \{x \in \R^d : f(x)<+\infty\}$ is nonempty), closed, and has at least one minimizer. We denote the set of all minimizers by $\cX_f$ and write $f_\star \eqdef \min_x f(x)$.

The dominant theoretical paradigm for such tasks is built around \emph{stationarity}. Convergence guarantees are typically stated in terms of the norm of the gradient, asserting that an algorithm reaches a point where $\norm{\nabla f(x)}$ is small. This framework underlies much of the modern theory of stochastic gradient methods.

From the modeling perspective, however, stationarity-based guarantees are fundamentally misaligned with the goals of learning. A stationary point can be a global minimizer, a poor local minimizer, or a saddle point; first-order information alone does not distinguish among these possibilities. Therefore, a small gradient provides little guidance on whether an algorithm meaningfully solves the underlying problem. Indeed, in highly non-convex landscapes, large regions of near-zero gradient may exist far from any global solution \citep{pascanu2014saddle}. Algorithms that follow gradient-based dynamics can stagnate in such regions, producing iterates that are theoretically certified as ``successful'' while being arbitrarily suboptimal in objective value \citep{jin2017escape}. In this sense, stationarity is a weak proxy for optimization progress.

Empirical evidence further challenges the relevance of gradient norms as a convergence signal. In long training runs of modern ML models, the gradient norm has been observed not to decrease, and in some cases to \emph{increase}, toward the end of training \citep{defazio2025gradients}. This behavior has been attributed to interactions between weight decay, learning rate schedules, and normalization layers.

In practice, escaping stationary points is usually left to stochasticity and noise, either explicitly (e.g., perturbed gradient methods \citep{jin2017escape, kleinberg2018alternative, jin2021nonconvex, ahn2023escape}) or implicitly (e.g., minibatch sampling \citep{smith2017bayesian}). Even when successful, the resulting guarantees of such methods remain fundamentally local. At best, they ensure convergence to approximate second-order stationary points, which still need not correspond to globally good solutions.

Thus, much of the existing non-convex optimization theory addresses the question: \emph{what local geometric properties characterize the iterates produced by an algorithm?} rather than the more fundamental question: \emph{when does an algorithm find a good solution?}
This state of affairs motivates the central question of this work:
\begin{center}
    \emph{How to obtain global convergence guarantees for non-convex problems without relying on gradient-based dynamics and gradient-norm metrics?}
\end{center}
Addressing it requires rethinking the algorithmic primitives and structural assumptions of non-convex optimization, moving toward mechanisms that are not driven solely by local first-order information.

\subsection{Ball Proximal Point Method}

To study the questions raised above, we revisit the Ball Proximal Point Method ({\newalg}) -- a conceptual framework for global optimization introduced by \citet{gruntkowska2025ball}.\footnote{Ball oracles have also been studied in several prior works \citep{carmon2020acceleration, carmon2021thinking, asi2021stochastic}; see \citet{gruntkowska2025ball} for a review.} {\newalg} is a trust-region–type method in which each iterate is obtained by minimizing the objective over a ball of radius $t_k>0$ centered at the current point:
\begin{align}\label{eq:bpm}
    \boxed{x_{k+1} \in \BProxSub{t_k}{f}{x_k}}
\end{align}
(in the remainder of this work, we restrict attention to the constant-radius setting and assume $t_k \equiv t > 0$ $\forall k$), where the \emph{ball-proximal (``broximal'') operator} is defined as
\begin{align*}
    \BProxSub{t}{f}{x} \eqdef \argmin_{z \in \cB_{\mX}(x, t)} f(z).
\end{align*}
Here $\cB_{\mX}(x, t) = \brac{z\in \R^d : \norm{z-x}_{\mX} \leq t}$ and $\norm{\cdot}_{\mX}$ is the norm induced by a symmetric positive definite matrix $\mX \in \R^{d\times d}$, defined by $\norm{x}_{\mX} \eqdef \sqrt{x^\top \mX x}$.\footnote{The original {\newalg} paper \citep{gruntkowska2025ball} considered the simplified setting where $\mX = \mI$ is the identity matrix. Other norms were considered in \citet{gruntkowska2025non}.} We denote $\inp{x}{y}_{\mX} \eqdef \inp{\mX x}{y} = x^\top \mX y$. When $\mX = \mI$, we omit the subscript and simply write $\cB(x, t)$, $\norm{\cdot}$ and $\inp{\cdot}{\cdot}$.

Like the classical Proximal Point Method (\algnamesmall{PPM}) \citep{rockafellar1976monotone}, {\newalg} solves an auxiliary optimization problem at each iteration, but with a hard trust-region constraint instead of a quadratic regularizer. The operator is set-valued in general, which is why we write ``$\in$’’ rather than ``$=$’’ in~\eqref{eq:bpm}. When multiple minimizers exist within the ball, the algorithm is free to choose any of them.

Motivated by challenges in non-convex optimization, the method was originally developed as an oracle capable of escaping local minima and saddle points. Along the way, it was found to enjoy unusually strong convergence guarantees: it converges \emph{linearly} without the strong convexity assumptions required by methods such as \algnamesmall{PPM}. Moreover, while \algnamesmall{PPM} with finite steps can only produce approximate solutions, {\newalg} can attain the \emph{exact global minimum} in a \emph{finite number of iterations} using \emph{finite radii}. Despite this promise, and despite evidence that the method can handle non-convexity (see \Cref{fig:nonconv}), existing analyses focus almost exclusively on convex objectives, with only limited non-convex treatment \citep{gruntkowska2025ball}.

Our aim is to understand what properties of the objective enable {\newalg} to achieve global convergence, and to identify the structural assumptions under which such guarantees can be established in the \emph{non-convex} setting. Crucially, we do \emph{not} view {\newalg} as a practical algorithm in its current form, nor do we claim efficient implementations for the problems we study. Indeed, the broximal oracle itself is very powerful and generally intractable: if the radius is sufficiently large ($t_0 \geq \norm{x_0-x_{\star}}_{\mX}$, where $x_0$ is a starting point and $x_{\star}$ is an optimal point), then $x_{\star} \in \BProxSub{t_0}{f}{x_0}$ and the algorithm finds a global solution in $1$ step. Thus, solving the subproblem can be as difficult as solving the original problem itself. While practical implementations are possible (see \Cref{sec:practice}), they are not the focus of this work. Rather, we treat {\newalg} as a conceptual tool for isolating and studying a novel class of non-convex functions for which global convergence can be established. This class stands on its own and may be viewed as an analytic tool for studying optimization algorithms beyond {\newalg} itself.

\section{Contributions}

This paper proposes a new framework for global non-convex optimization beyond stationarity. We introduce a novel function class under which {\newalg} converges globally, and show that this condition strictly subsumes several established non-convex frameworks.
Our main contributions are:

$\triangleright$ \textbf{New assumptions for non-convex optimization.} We introduce a novel class of non-convex functions, termed \emph{Broximal Aligned Functions}, defined by Assumptions \ref{as:main1} and \ref{as:main2}. The framework enables global optimization guarantees beyond gradient-based reasoning.

$\triangleright$ \textbf{Global convergence of BPM.} We prove that {\newalg} converges to a global minimizer within this class, without reliance on gradients or stochastic noise (\Cref{thm:nc_bpm}).

$\triangleright$ \textbf{Interpretation and examples.} We provide intuitive interpretations, comparisons to existing function classes, and concrete examples demonstrating that our assumptions are weaker than several standard conditions such as pseudoconvexity, quasiconvexity, quasar convexity, and the aiming condition (\Cref{sec:examples}).

The remainder of the paper is organized as follows. In \Cref{sec:nonocnv_def}, we introduce the new function class via Assumptions \ref{as:main1} and \ref{as:main2} and motivate its origins. Building on this abstract framework, \Cref{sec:conv_thm} presents global convergence guarantees for {\newalg}. \Cref{sec:examples} then clarifies the assumptions by providing examples and comparisons to other non-convex function classes. We provide comments on practical implementations in \Cref{sec:practice} and conclude in \Cref{sec:summary}.

We emphasize that this work is theoretical and foundational in nature. Our goal is not to propose a new practical optimizer, but to broaden the conceptual foundations of non-convex optimization by identifying mechanisms that enable global guarantees beyond gradient-based paradigms. We hope that this perspective will stimulate further research into practical approximations and connections to modern ML optimization methods.

All proofs are deferred to the Appendix, and a complete list of notation is provided in \Cref{table:notation}.

\section{Optimization Beyond Convexity}\label{sec:nonocnv_def}

Non-convex optimization is generally challenging, yet algorithmic developments have yielded both heuristic and provably convergent approaches. One pragmatic approach exploits the local nature of many optimization methods by running them in parallel from multiple initializations and aggregating the resulting solutions. Another line of work studies particle-based approaches \citep{fornasier2020consensus, fornasier2024consensus, fornasier2025regularity}, which address non-convex problems by reformulating them as PDE-governed dynamics over measures. Although originally heuristic, these methods now admit provable global solutions. The main drawback is computational: approximating the PDE typically requires an exponentially large or even infinite number of particles.
In addition, certain problems become tractable when the objective has some special structure. For example, one-layer neural networks have no spurious local minima \citep{feizi2017porcupine, haeffele2017global, soltanolkotabi2018theoretical}, and \algnamesmall{SGD} can find global minima in linear and sufficiently wide over-parameterized networks under appropriate assumptions \citep{danilova2022recent, shin2022effects, allen2019convergence}. Similarly, with suitable assumptions on input data, a range of non-convex problems, including matrix completion, some deep learning tasks, phase retrieval, and eigenvalue problems, display ``convexity-like'' behavior \citep{ge2016matrix, ge2017no, bartlett2018gradient, kangal2018subspace}.

Despite the strong theoretical guarantees offered by convex optimization, its assumptions are often too restrictive for modern large-scale models. This has motivated the development of intermediate frameworks that retain meaningful convergence properties while being applicable to broader problem classes. Examples include star convexity \citep{nesterov2006cubic}, quasar convexity \citep{hardt2018gradient, hinder2019near}, pseudoconvexity \citep{mangasarian1975pseudo}, quasiconvexity \citep{arrow1961quasi}, invexity \citep{craven1985invex}, variational coherence \citep{zhou2017stochastic}, the Polyak-{\L}ojasiewicz (P{\L}) condition \citep{polyak1963gradient, lojasiewicz1963topological}, and aiming condition \citep{liu2023aiming}. Each of these offers certain benefits, but also comes with inherent limitations and applicability constraints.

Here, we move beyond settings in which global convergence follows from the absence of suboptimal local minima, and instead pursue the more ambitious goal of achieving global optimality without relying on such favorable geometric properties. This shift requires abandoning first-order methods in favor of more powerful ball oracles. Drawing on the theory of {\newalg}, we reverse-engineer the underlying conditions necessary to achieve global convergence under this broader framework.

\subsection{The Origins}

The original analysis of {\newalg} relies on the \emph{Second Brox Theorem} \citep{gruntkowska2025ball}, exploiting the fact that if $f$ is convex, then there exists $c_t(x) \geq 0$ such that $$f(y) - f(u) \geq c_t(x)\inp{x - u}{y - u}_{\mX}$$ for $x \in \dom f$, $u \in \BProxSub{t}{f}{x}$ and any $y \in \R^d$. In particular, specializing to $y=x_{\star}$, where $x_\star$ is a global minimizer, yields $c_t(x)\inp{x - u}{u - x_{\star}}_{\mX} \geq f(u) - f_{\star} \geq 0$. This observation motivates dropping convexity altogether and instead postulating this inequality directly, leading to the following assumption (illustrated in \Cref{fig:as_main_geo}):
\begin{assumption}\label{as:main1}
    Let $f:\R^d \to \R \cup \{+\infty\}$ be a proper, closed function, and let $t>0$. There exists $x_{\star} \in \cX_f$ such that, for every $x \in \dom f \backslash \brac{\cX_f + \cB_{\mX}(0,t)}$
    \begin{align}\label{eq:as_main}
        \inp{x-u}{x_{\star}-u}_{\mX} \leq 0
    \end{align}
    for any $u \in \BProxSub{t}{f}{x}$.
\end{assumption}
Intuitively, whenever $x$ lies outside of the $t$-neighborhood of the set of global minimizers $\cX_f$, any local minimizer $u$ of $f$ within the ball $\cB_{\mX}(x, t)$ cannot lie ``past'' some global minimizer $x_{\star}$, ensuring that the step from $x$ to $u$ is aligned toward it.
The relaxation $\textnormal{dist}(x, \cX_f) > t$ is essential: if the condition were imposed for all $x \in \dom f$, then Lemma \ref{lemma:2ptdecr} would imply uniqueness of the global minimizer. This phenomenon is consistent with the convex analysis of \citet{gruntkowska2025ball}, where the constant $c_t(x)$ becomes $0$ when $u\in\cX_f$, making the inner product sign irrelevant.

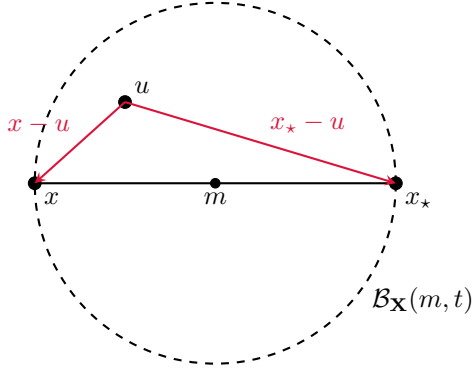
\begin{figure}[t]
\centering
\begin{tikzpicture}[scale=1.2]
    \coordinate (x) at (-2,0);
    \coordinate (xs) at (2,0);
    \coordinate (m) at (0,0);  
    \coordinate (u) at (-1,0.9);

    \draw[thick, dashed] (m) circle (2);

    \filldraw[black] (x) circle (2pt) node[below right] {$x$};
    \filldraw[black] (xs) circle (2pt) node[below right] {$x_\star$};
    \filldraw[black] (u) circle (2pt) node[above right] {$u$};
    \filldraw[black] (m) circle (1.5pt) node[below] {$m$};

    \draw[thick] (x) -- (xs);

    \draw[->, thick, >=stealth, CrimsonRed] (u) -- (x) node[midway, above left] {$x-u$};
    \draw[->, thick, >=stealth, CrimsonRed]  (u) -- (xs) node[midway, above right] {$x_\star-u$};

    \node at (2.3,-1.3) {$\mathcal{B}_{\mX}(m,t)$};
\end{tikzpicture}
\caption{The inequality in \Cref{as:main1} is equivalent to requiring that $u$ lies inside the ball centered at the midpoint $m=\frac{1}{2} (x+x_\star)$ with radius $t=\frac{1}{2}\norm{x-x_\star}_\mX$. Equivalently, the angle at $u$ in the triangle formed by $x$, $u$ and $x_\star$ is $\geq 90$ degrees (in the geometry induced by $\mX$).}
\label{fig:as_main_geo}
\end{figure}

Under \Cref{as:main1}, unlike in the convex case, the broximal operator need not be single-valued outside $\cX_f + \cB_{\mX}(0,t)$. In such cases, we can define $\BProxSub{t}{f}{x}$ to select a maximally distant point from $x$ among the broximal minimizers. In particular, if $x \in \BProxSub{t}{f}{x}$ but the set contains additional elements, the algorithm must choose a point in $\BProxSub{t}{f}{x}\setminus\{x\}$ to ensure progress.
A more problematic case arises when the only broximal minimizer is $x$ itself, i.e., $\BProxSub{t}{f}{x}=\{x\}$. Under \Cref{as:main1}, this does not necessarily imply that $x\in\cX_f$, and we will need an additional assumption to rule out this scenario. First, however, we examine what can be deduced from \Cref{as:main1} alone.

Suppose that $u \in \BProxSub{t}{f}{x} \not\subseteq \cX_f$ is such that $f(u) = f(x)$ and $u\neq x$. By \Cref{lemma:2ptdecr}, this implies $x \not\in \BProxSub{t}{f}{u}$. Since $\BProxSub{t}{f}{u}$ is non-empty (\Cref{thm:1stbrox_nc}), there exists $\bar{u} \in \BProxSub{t}{f}{u}$ such that $f(\bar{u})<f(u)$ (the inequality is strict because otherwise we would have $x \in \BProxSub{t}{f}{u}$). Consequently, $\bar{u} \not\in \cB_{\mX}(x, t)$, meaning that $\norm{x-\bar{u}}_{\mX} > t$.
Therefore, under \Cref{as:main1}, as long as the algorithm makes progress ($x_k \neq x_{k+1}$), a non-decreasing step $f(x_k) = f(x_{k+1})$ must be followed by a strict decrease $f(x_{k+1}) > f(x_{k+2})$, and then $\norm{x_k-x_{k+2}}_{\mX} > t$. This behavior is formalized in \Cref{lemma:t_dist}.

The above argument relies on the condition $x_k \neq x_{k+1}$. To guarantee this unless a global minimizer has been reached, we introduce it explicitly as an assumption.
\begin{assumption}\label{as:main2}
    Let $f:\R^d\to\R \cup \{+\infty\}$ be a proper, closed function, and let $t>0$. Then
    \begin{align*}
        \BProxSub{t}{f}{x} = \{x\} \quad\implies\quad x\in\cX_f
    \end{align*}
    for all $x \in \dom f$.
\end{assumption}
In other words, if~$x$ is a strict minimizer of~$f$ over the ball $\cB_{\mX}(x, t)$, then~$x$ must already be globally optimal, i.e., there are no ``bad'' local minima detectable at the scale of radius~$t$.
Note that \Cref{as:main2} still permits $x \in \BProxSub{t}{f}{x}$ with $x \not\in \cX_f$, provided that there exists another point $\bar{x} \in \BProxSub{t}{f}{x}$ such that $\bar{x} \neq x$.

To simplify references, we introduce a notation for the family of functions that satisfy our new assumptions.
\begin{definition}[Broximal Aligned Functions]\label{def:BA}
    A proper, closed function $f:\R^d\to\R \cup \{+\infty\}$ belongs to the class $\newclass{\mX}{t}$ if it satisfies Assumptions \ref{as:main1} and \ref{as:main2} for a radius $t>0$ with respect to the geometry induced by the symmetric positive definite matrix $\mX \in \R^{d\times d}$.
\end{definition}
Importantly, this class requires no convexity, differentiability, smoothness, or Lipschitz assumptions -- only the geometric alignment condition of \Cref{as:main1} and the absence of isolated local minima as specified by \Cref{as:main2}.

\section{Convergence Theory}\label{sec:conv_thm}

It turns out that Assumptions \ref{as:main1} and \ref{as:main2} are enough for global convergence guarantees. Building upon them, the main convergence result relies on the fact that every \emph{second} iterate is separated by a distance of at least~$t$.

\begin{restatable}{theorem}{MAINTHM}\label{thm:nc_bpm}
    Assume that $f\in\newclass{\mX}{t}$ and let $\{x_k\}_{k\geq0}$ be the iterates of \algnamesmall{BPM} run with any $t>0$, where $x_0\in \dom f$. Then 
    \begin{enumerate}[label=(\roman*)]
        \item $\cX_f\cap \cB_{\mX}(x_k, t)\neq\emptyset$ if and only if $x_{k+1}$ is optimal.\label{pt:nc_opt}
        
        \item \label{pt:nc_dist_decr_bpm}
        The distances $\brac{\norm{x_k - x_{\star}}_{\mX}}_{k\geq0}$ are non-increasing. In particular, if $\cX_f\cap \cB_{\mX}(x_k, t)=\emptyset$, we have
        \begin{align*}
            \norm{x_{k+1} - x_{\star}}_{\mX}^2 \leq \norm{x_{k-2} - x_{\star}}_{\mX}^2 - \frac{t^2}{3}.
        \end{align*}
        
        \item \label{pt:nc_dist_bpm_lin}
        If $\flr{\frac{K}{3}} \geq \frac{3}{t^2} \norm{x_0 - x_{\star}}_{\mX}^2$, then $x_K\in\cX_f$.
    \end{enumerate}
\end{restatable}
Unlike the non-convex optimization literature standard, the guarantee above is not stated in terms of stationary points, but rather in terms of distances from the minimizer. We show that after a sufficient number of iterations, the method finds the exact global minimum.

\section{Unpacking the Broximal Geometry}\label{sec:examples}

So far, we have operated with the abstract function class $\newclass{\mX}{t}$ without clarifying which objective functions belong to it. At first glance, \Cref{as:main1} may seem cryptic, as it cannot be expressed using standard smoothness or curvature conditions. This section aims to unpack the structure behind our assumptions and explain the class of functions they define. We begin with simple examples that illustrate the kinds of functions covered by our framework. We then present a study of existing notions of generalized convexity, their characterizations, and their relationship to our assumptions (summarized in \Cref{tab:summary} and \Cref{fig:assumption_hierarchy}).

Before proceeding, we emphasize an important caveat: in principle, \emph{any} function can be cast under our framework for sufficiently large $t$. Once the ball defining the broximal operator contains a global minimizer $x_\star$, the assumptions hold trivially. Of course, this regime is not of practical interest: as $t\to\infty$, the optimization subroutine in \eqref{eq:bpm} becomes equivalent to global optimization, and is therefore no easier than the original problem. For this reason, we focus on settings where the assumptions hold for finite $t$.

Importantly, for \emph{any} fixed $t>0$, the standard non-convex function classes discussed below are more restrictive than $\newclass{\mI}{t}$. In other words, our framework is strictly more general and can capture broader non-convex behavior. While the presence of $\mX$ adds additional flexibility, we restrict attention to the standard Euclidean case $\mX=\mI$ for the remainder of this section, unless stated otherwise.

\subsection{Toy Examples}\label{sec:toy_examples}

We first present a simple illustrative example that satisfies Assumptions \ref{as:main1} and \ref{as:main2} while exhibiting a high degree of non-convexity.

\begin{example}\label{example:1}
    Consider the one-dimensional function
    \begin{align*}
    f(x) = |x| + 10 \sin(x),
    \end{align*}
    which has infinitely many local minima spaced every $2\pi$ (see \Cref{fig:nonconv}). Hence, for any fixed radius $t \geq 2 \pi$, the ball $\cB_{\mX}(x, t)$ always contains at least two local minimizers whose function values differ. It is easy to check that both Assumptions \ref{as:main1} and \ref{as:main2} hold for this radius, and so $f\in\newclass{\mI}{t}$ for any $t \geq 2 \pi$.
\end{example}

\begin{figure}[t]
    \centering
    \includegraphics[width=\columnwidth]{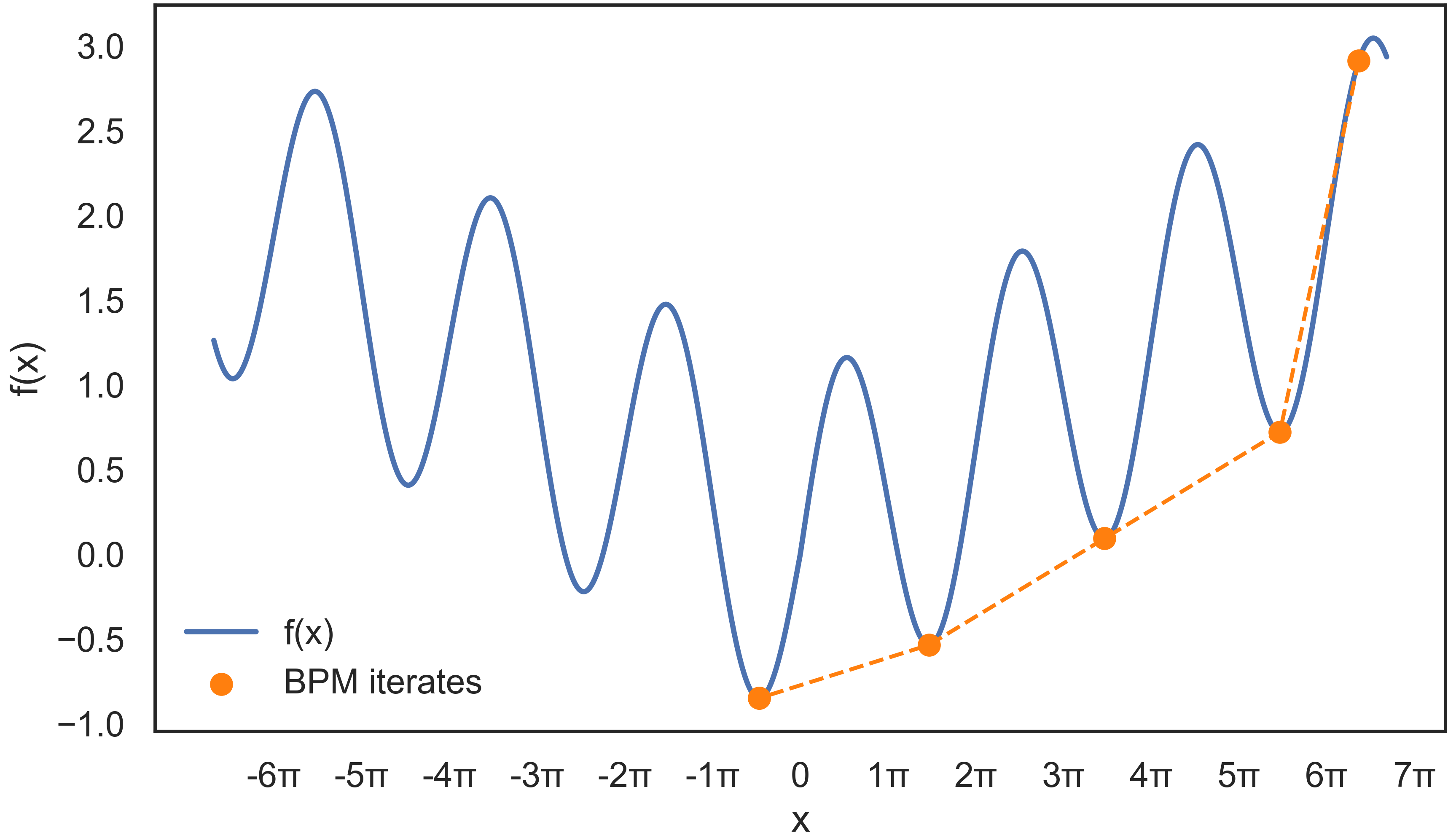}
    \caption{Trajectory of {\sf BPM} with $t=2\pi$ starting from $x_0=20$ applied to $f(x) = |x| + 10 \sin(x)$ (\Cref{example:1}).}
    \label{fig:nonconv}
\end{figure}

The example above shows that our setup does not preclude the existence of multiple local minima. This is in contrast to classical conditions, under which all stationary points are globally optimal (see \Cref{tab:summary}). The function class we identify can exhibit multiple local minimizers and nontrivial geometry, while still permitting global convergence of {\newalg}.

The next example shows that the set of minimizers $\cX_f$ can be disconnected.

\begin{restatable}{example}{EXAMPLE}\label{example:2}
    Fix $a \in \R^d$ such that $a \neq 0$ and consider the function $f:\R^d \mapsto \R \cup \brac{+\infty}$ defined by 
    \begin{align*}
    f(x) = \begin{cases}
        \norm{x}^2 & x \neq a, \\
        0 & x = a.
    \end{cases}
    \end{align*}
    Then $f$ has a disconnected set of minimizers $\cX_f = \brac{0, a}$, but one can show that $f\in\newclass{\mI}{t}$ with any $t>0$.
\end{restatable}

The example above motivates the construction of a larger class of functions that satisfy Assumptions \ref{as:main1} and \ref{as:main2} (see \Cref{lemma:constructing}).

We emphasize that both Examples \ref{example:1} and \ref{example:2} belong to $\newclass{\mX}{t}$ for suitable choices of $t$, yet they fail to satisfy any of the classical conditions listed in \Cref{tab:summary}.

\begin{table*}[t]
    \renewcommand{\arraystretch}{1.8}
    \caption{\small
    \textbf{Comparison of convex and non-convex function classes.} 
    \textbf{Non-Convex} = does not require $f$ to be convex;
    \textbf{Non-Differentiable} = does not require $f$ to be differentiable;
    \textbf{Stationary pts} = allows existence of stationary points that are not global minima;
    \textbf{Non-unique min} = allows existence of multiple global minima;
    \textbf{Disconnected $\cX_f$}: the set of minimizers need not be connected.}
    \label{tab:summary}
    \scriptsize
    \centering
    \begin{adjustbox}{width=\textwidth,center}
    \begin{threeparttable}
    \begin{tabular}{c c c c c c c}
        \toprule
        \textbf{Function class}
        & \textbf{Non-Convex}
        & \textbf{Non-Differentiable}
        & \textbf{Stationary pts}
        & \textbf{Non-unique min}
        & \textbf{Disconnected $\cX_f$}
        \\
        \midrule
        \bf Strongly convex & \xmark & \cmark & \xmark & \xmark & \xmark \\
        \bf Convex & \xmark & \cmark & \xmark & \cmark & \xmark \\
        \makecell{\bf Pseudoconvex \\ {\tiny \citep{mangasarian1975pseudo}}} & \cmark & \xmark & \xmark & \cmark & \xmark \\
        \makecell{\bf Quasiconvex \\ {\tiny \citep{arrow1961quasi}}} & \cmark & \cmark & \cmark & \cmark & \xmark \\
        \makecell{\bf Star-convex \\ {\tiny \citep{nesterov2006cubic}}} & \cmark & \xmark & \xmark & \cmark & \xmark \\
        \makecell{\bf Quasar convex \\ {\tiny \citep{hardt2018gradient}}} & \cmark & \xmark & \xmark & \cmark & \xmark \\
        \makecell{\bf Aiming \\ {\tiny \citep{liu2023aiming}}} & \cmark & \xmark & \xmark & \cmark & \xmark \\
        \makecell{\bf Polyak-{\L}ojasiewicz \\ {\tiny \citep{polyak1963gradient, lojasiewicz1963topological}}} & \cmark & \xmark & \xmark & \xmark & \xmark \\
        \makecell{\bf Quadratic growth \\ {\tiny \citep{bonnans1995second}}} & \cmark & \cmark & \cmark & \cmark & \cmark \\
        \midrule
        \makecell{$\newclass{\mX}{t}$ \\ (ours)} & \cmark & \cmark & \cmark & \cmark & \cmark \\
        \bottomrule
    \end{tabular}
    \end{threeparttable}
    \end{adjustbox}
\end{table*}

\begin{figure}[t]
\centering
\resizebox{\columnwidth}{!}{%
    \begin{tikzpicture}[
      box/.style={
        draw,
        rectangle,
        rounded corners,
        minimum width=2.5cm,
        minimum height=0.9cm,
        align=center
      },
      impl/.style={double, ->, thick, >=latex, shorten >=2pt},
      implcolor/.style={
        impl,
        draw=CrimsonRed
      },
      noimpl/.style={dashed, thick},
      shorten >=2pt,
      shorten <=2pt,
      x=4.2cm,
      y=1.6cm
    ]
    
    
    \node[box] (SC) at (0,0) {Strongly\\Convex};
    
    \node[box] (C)  at (0,-1) {Convex};
    \node[box] (PL) at (1,-1) {Polyak--\\{\L}ojasiewicz};
    
    \node[box] (PC) at (-1,-2) {Pseudoconvex};
    \node[box] (QC) at (0,-2)  {Quasiconvex};
    \node[box] (QG) at (1,-2)  {Quadratic\\Growth};
    
    \node[box] (Aiming) at (-1,-3) {Aiming};
    \node[box] (Star)   at (0,-3)  {Star-Convex};
    \node[box] (QC2)    at (1,-3)  {Quasar\\Convex};
    
    \node[box] (PA) at (-0.45,-4.3) {Broximal\\Aligned};
    
    
    \draw[impl] (SC) -- (C);
    \draw[impl] (SC) -- (PL);
    
    \draw[impl] (C) to[out=180,in=90] (PC);
    \draw[impl] (C) to[out=-90,in=90]  (QC);
    \draw[impl] (C) to[out=-10,in=10]  (Star);
    \draw[impl] (C) to[out=180,in=0]  (Aiming);
    \draw[impl] (PC) to[out=0,in=-180]  (QC);
    
    \draw[impl] (Star) to[out=0,in=180] (QC2);
    \draw[impl] (PL)   to[out=-90,in=90] (QG);
        
    \draw[implcolor] (QC) to[out=-150,in=90] (PA);
    \draw[implcolor] (QC2) to[out=-90,in=0] (PA);
    \draw[implcolor] (Aiming) to[out=-90,in=180] (PA);
        
    \draw[noimpl] (QC)     to[out=0,in=180] (QC2);
    \draw[noimpl] (PC)     to[out=-30,in=150] (QC2);
    \draw[noimpl] (Aiming) to[out=-20,in=-160] (QC2);
    \draw[noimpl] (QG)     to[out=-90,in=90] (QC2);
    \draw[noimpl] (C)      to[out=0,in=180] (PL);
    
    \end{tikzpicture}
}
\caption{\textbf{Hierarchy of assumptions.} Solid arrows encode one-way implications, dashed lines denote mutual non-implications, and {\color{CrimsonRed}red} arrows indicate conditions guaranteeing Broximal Alignment for any $t>0$. For simplicity, we assume $f$ is $L$-smooth (i.e., $\norm{\nabla f(x) - \nabla f(y)} \leq L \norm{x-y}$ $\forall x, y \in\R^d$) with domain $\dom f = \R^d$. See the appendix and \citet{karimi2016linear, hinder2019near, khanh2025star} for proofs.}
\label{fig:assumption_hierarchy}
\end{figure}
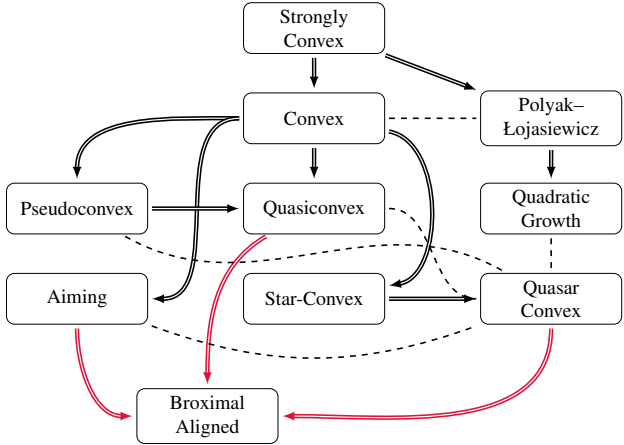

\subsection{Quasiconvexity}

Quasiconvexity is one of the most classical and widely studied relaxations of convexity, dating back to early work by \citet{arrow1961quasi}. It plays an important role in optimization theory because it preserves several global properties of convex functions while allowing for richer geometric structure. As a result, quasiconvex optimization has found applications in many areas, including economics \citep{wolfstetter1999topics, laffont2002theory}, control \citep{seiler2010quasiconvex}, computational geometry \citep{eppstein2005quasiconvex}, and computer vision \citep{ke2007quasiconvex}.

\begin{restatable}[Quasiconvexity]{definition}{QUASICONVDEF}\label{def:quasiconvex}
    A function $f: \R^d \to \R \cup \{+\infty\}$ is \emph{quasiconvex} if
    \begin{align*}
        f\parens{(1-\lambda)x + \lambda y} \leq \max\brac{f(x), f(y)}
    \end{align*}
    for all $x, y \in\R^d$ and $\lambda \in [0,1]$. If the inequality is strict for all $x \neq y$ and $\lambda \in (0,1)$, then $f$ is \emph{strictly quasiconvex}.
\end{restatable}
\Cref{def:quasiconvex} can be viewed as a relaxation of convexity obtained by replacing the convex combination of function values by a maximum. A standard consequence is that all sublevel sets of a quasiconvex function are convex (\Cref{lemma:quasi_conv_lf_conv}).

The following result shows that strict quasiconvexity is sufficient for membership in our function class.
\begin{restatable}{theorem}{QUASITHM}\label{thm:quasi_conv_as}
    Suppose that $f$ is strictly quasiconvex. Then $f\in\newclass{\mI}{t}$ for any $t>0$.
\end{restatable}
This inclusion holds for \emph{every} radius $t>0$, showing that $\newclass{\mI}{t}$ contains the class of strictly quasiconvex functions for any fixed~$t$.
At the same time, unlike functions in our class, quasiconvex functions remain restrictive for many modern non-convex objectives in ML. The requirement that every sublevel set be convex excludes functions with complex basin structures, multiple disconnected minima, or other nontrivial geometric features that arise in deep learning and related applications.

\subsection{Pseudoconvexity}

Pseudoconvexity is a first-order relaxation of convexity that occupies the middle ground between convex and quasiconvex functions. Introduced in early work by \citet{mangasarian1975pseudo}, it has since been studied extensively in nonlinear programming \citep{mangasarian1994nonlinear, bazaraa1993nonlinear}.

\begin{restatable}[Pseudoconvexity]{definition}{PSEUDOCONVDEF}\label{def:pseudoconvex}
A differentiable function $f:\R^d \to \R$ is \emph{pseudoconvex} if
\begin{align*}
    \inp{\nabla f(x)}{y - x} \geq 0 \quad \implies \quad f(y) \geq f(x)
\end{align*}
for all $x,y \in \R^d$.
\end{restatable}

Every convex differentiable function is pseudoconvex, but the converse does not hold. Nevertheless, pseudoconvexity still guarantees that every stationary point is a global minimizer, and a classical result further shows that any differentiable pseudoconvex function is quasiconvex. This makes pseudoconvexity a natural setting in which to test our broximal alignment conditions.

\begin{restatable}{theorem}{PSEUDOTHM}\label{thm:pseudo_conv_as}
    Suppose that $f$ is pseudoconvex. Then $f \in \newclass{\mI}{t}$ for any $t>0$.
\end{restatable}

Although \Cref{def:pseudoconvex} is stated for differentiable functions, in \Cref{sec:gen_pseudoconv} we extend the comparison to nondifferentiable generalizations of pseudoconvexity and show that our assumptions remain strictly weaker in that broader setting as well.
   
\subsection{Quasar Convexity}

Quasar convexity \citep{hardt2018gradient, hinder2019near}, like quasiconvexity and pseudoconvexity, generalizes the notion of unimodality to higher dimensions. Informally, quasar convex functions are unimodal along all lines passing through a global minimizer \citep{hinder2019near}. In dimensions greater than one, quasar convexity is neither implied by nor implies quasiconvexity or pseudoconvexity.

\begin{restatable}[Quasar convexity]{definition}{QUASARCONVDEF}
    A differentiable function $f: \R^d \to \R$ is $\zeta$-quasar convex with respect to $x_{\star}\in\cX_f$ for some $\zeta\in(0,1]$ if
    \begin{align*}
        f(x) - f_\star \leq \frac{1}{\zeta} \inp{\nabla f(x)}{x - x_{\star}}
    \end{align*}
    for all $x \in\R^d$.
\end{restatable}
The class is parameterized by a constant $\zeta\in(0,1]$ which controls the degree of non-convexity. As $\zeta$ becomes smaller,~$f$ becomes ``more non-convex''. 
Quasar convexity appears in several practical settings, such as learning linear dynamical systems \citep{hardt2018gradient}, a problem closely related to training recurrent neural networks.
When $\zeta=1$, the above condition is called \textit{star convexity} \citep{nesterov2006cubic}, another well-known relaxation of convexity that has been hypothesized to describe the loss geometry of neural networks in large neighborhoods around global minimizers \citep{kleinberg2018alternative, zhou2019sgd}.

As the next result shows, quasar convex objectives also fall within the scope of our framework.
\begin{restatable}{theorem}{QUASARTHM}\label{thm:neg_inp_quasar}
    Suppose that $f$ is $\zeta$-quasar convex with respect to $x_{\star}\in\cX_f$. Then $f \in \newclass{\mI}{t}$ for any $t>0$.
\end{restatable}

\subsection{Aiming Condition}

The \emph{aiming condition} \citep{liu2023aiming} links the gradient direction to the function value, while allowing the reference minimizer to vary with the current point. Following \citet{liu2023aiming}, we assume throughout that $f_\star = 0$. We denote by $\Pi(\cdot, \cX)$ the projection onto a set $\cX$ (i.e., $\Pi(x, \cX) \eqdef \argmin_{y \in \cX} \norm{x-y}_2^2$).

\begin{restatable}[Aiming Condition]{definition}{AIMINGDEF}\label{def:aiming}
    A differentiable function $f: \R^d \to \R$ satisfies the \emph{aiming condition} if $f_\star = 0$ and there exist $\bar{x} \in \Proj{\cX_f}{x}$ and $\theta > 0$ such that
    \begin{align}\label{eq:aiming}
        \theta f(x) \leq \inp{\nabla f(x)}{x - \bar{x}}
    \end{align}
    for all $x \in \R^d$.
\end{restatable}

The aiming condition is closely related to quasar convexity: while quasar convexity requires \eqref{eq:aiming} to hold with respect to a fixed minimizer $x_\star\in\cX_f$, aiming permits the reference point $\bar{x}$ to depend on $x$.
This yields a strictly larger class of functions, encompassing, for example, wide neural network losses \citep{liu2023aiming}. Even so, aiming remains stronger than our assumptions.

\begin{restatable}{theorem}{AIMINGTHM}\label{thm:aiming}
    Let $f$ satisfy the aiming condition and have a unique global minimizer. Then $f \in \newclass{\mI}{t}$ for any $t>0$.
\end{restatable}

Here, we additionally require the minimum to be unique. However, \citet{liu2023aiming} do not rely on the aiming condition in isolation, but rather combine it with quadratic growth \citep{necoara2019linear}, smoothness, and interpolation assumptions.

\subsection{Operations that Preserve Assumptions~\ref{as:main1} and~\ref{as:main2}}

Lastly, we record several basic closure properties of
Assumptions~\ref{as:main1} and~\ref{as:main2}. In these results, we again work under the general conditions formulated in the geometry induced by $\mX$. These properties clarify which transformations leave the geometric structure underlying our framework unchanged, and can be useful for constructing new admissible objectives from existing ones.

We begin with monotone transformations of the objective.

\begin{restatable}{lemma}{MONOTRANS}\label{lemma:transformation-11}
    Let $h : \R^d \to \R \cup \{+\infty\}$ be a proper, closed function with a nonempty set of minimizers $\cX_h$ and let $g : \R \to \R$ be strictly increasing.
    Define the extended-real-valued function $f = g \circ h$ by
    \begin{align*}
        f(x) =
        \begin{cases}
        g(h(x)) & \text{if } h(x) \in \R, \\
        +\infty & \text{if } h(x) = + \infty.
     \end{cases}
    \end{align*}
    If $h$ satisfies Assumptions~\ref{as:main1} and~\ref{as:main2} for some $t > 0$, then $f$ also satisfies
    Assumptions~\ref{as:main1} and~\ref{as:main2} with the same $t$.
\end{restatable}

Next, we consider changes of coordinates, showing that our assumptions are invariant under transformations that preserve the geometry induced by $\mX$.

\begin{restatable}[$\mX$--orthogonal affine transformations]{lemma}{EUCISO}
    Let $h : \R^d \to \R \cup \{+\infty\}$ be a proper, closed function with nonempty set of minimizers $\cX_h \neq \emptyset$. Suppose that $h$ satisfies Assumptions~\ref{as:main1} and~\ref{as:main2} for some $t>0$.
    Let $\phi : \R^d \to \R^d$ be defined by
    \begin{align*}
        \phi(x) = \mQ x + b,
    \end{align*}
    where $\mQ \in \R^{d\times d}$ is invertible and satisfies $\mQ^\top \mX \mQ = \mX$.
    Define $f : \R^d \to \R \cup \{+\infty\}$ by
    \begin{align*}
        f(x) = \begin{cases}
        h(\phi(x)) & \quad x \in \phi^{-1}(\dom h), \\
        + \infty & \quad \text{otherwise}.
     \end{cases}
    \end{align*}
    Then $f$ is proper, has nonempty minimizer set $\cX_f = \phi^{-1}(\cX_h)$, and satisfies Assumptions~\ref{as:main1} and~\ref{as:main2} with the same radius $t$.
\end{restatable}

Finally, we note that simple affine transformations of the objective value itself also preserve the assumptions.

\begin{restatable}[Affine transformations of the function]{lemma}{AFFFUNC}
    Let $g : \R^d \to \R \cup \{+\infty\}$ be a proper, closed function with nonempty minimizer set $\cX_g$, satisfying Assumptions~\ref{as:main1} and~\ref{as:main2} for radius $t>0$. Let $a > 0$ and $b \in \R$, and define
    \begin{align*}
        f(x) = a\, g(x) + b.
    \end{align*}
    Then $f$ is proper, has nonempty minimizer set $\cX_f = \cX_g$, and satisfies Assumptions~\ref{as:main1} and~\ref{as:main2} with the same radius $t$.
\end{restatable}

\section{Practical Considerations}\label{sec:practice}

Our theoretical developments assume access to the broximal oracle, that is, the ability to exactly minimize $f$ over balls of the form $\cB_{\mX}(x, t)$. For many objectives of practical interest, particularly in modern ML, such subproblems are computationally intractable. Nevertheless, the framework remains relevant in practice through approximation strategies.

The first approach is to \emph{approximately solve the broximal subproblem} itself. For example, one may consider sampling-based approaches or apply an iterative optimization method to minimize the original objective $f$ over the constraint set $\cB_{\mX}(x_k,t_k)$ up to a prescribed accuracy. This strategy preserves the structure of the original problem, at the cost of additional inner-loop computation.

The second approach is to \emph{replace the original objective by a simpler local model}. Typical choices include first-order or quadratic approximations, such as linearizing $f$ around the current iterate and solving the resulting trust-region subproblem. Such model-based subproblems are often tractable and, in certain cases, admit closed-form solutions.

The effectiveness of these approximations is not limited to Euclidean settings. On the contrary, their benefits often become more pronounced in non-Euclidean geometries, where the norm $\norm{\cdot}_{\mX}$ is replaced by a certain operator norm.
This viewpoint underlies several modern optimization methods. In particular, the \algnamesmall{Muon} optimizer \citep{jordan2024muon} can be interpreted as performing approximate broximal steps by minimizing linearized local models over non-Euclidean balls induced by the spectral norm. We refer the reader to \citet{gruntkowska2025non} for a more detailed discussion and further connections.

\section{Summary and Outlook}\label{sec:summary}

This paper introduces a new analytic tool for global non-convex optimization that avoids the limitations of stationarity-based theory. Revisiting the Ball Proximal Point Method ({\newalg}), we identify a novel geometric condition --Broximal Alignment -- that guarantees global convergence without any convexity, smoothness, or Lipschitz assumptions. The resulting function class, $\newclass{\mX}{t}$, is unusually broad: it allows disconnected minimizer sets, multiple local minima, and highly non-convex landscapes, yet still ensures that {\newalg} reaches a global minimizer in finitely many iterations.
Our theory is strong in two ways. First, it provides global optimality guarantees in settings where gradient-based methods do not certify meaningful progress. Second, it strictly extends several classical non-convex frameworks, showing that these are all special cases of our geometric alignment principle. This positions our work as a conceptual foundation for global non-convex optimization.

At the same time, the framework is intentionally conceptual: the broximal oracle is generally intractable, so our results do not directly yield efficient algorithms. This opens a rich direction for future work: developing practical approximations of broximal steps to bridge the gap between our global theory and modern ML practice.

\section*{Acknowledgements}

This work was supported by funding from King Abdullah University of Science and Technology (KAUST): i) KAUST Baseline Research Scheme, ii) Center of Excellence for Generative AI (award no. 5940). K.G. is supported by the Google PhD Fellowship.





\bibliography{example_paper}

@article{nesterov2006cubic,
  title={Cubic regularization of Newton method and its global performance},
  author={Yurii Nesterov and Boris Polyak},
  journal={Mathematical Programming},
  year={2006},
  volume={108},
  pages={177-205}
}

@inproceedings{bartlett2018gradient,
  title={Gradient descent with identity initialization efficiently learns positive definite linear transformations by deep residual networks},
  author={Bartlett, Peter and Helmbold, Dave and Long, Philip},
  booktitle={International conference on machine learning},
  pages={521--530},
  year={2018},
  organization={PMLR}
}

@article{hinder2019near,
  title={Near-optimal methods for minimizing star-convex functions and beyond},
  author={Hinder, Oliver and Sidford, Aaron and Sohoni, Nimit S},
  journal={arXiv preprint arXiv:1906.11985},
  year={2019}
}

@inproceedings{kleinberg2018alternative,
  title={An alternative view: When does SGD escape local minima?},
  author={Kleinberg, Bobby and Li, Yuanzhi and Yuan, Yang},
  booktitle={International conference on machine learning},
  pages={2698--2707},
  year={2018},
  organization={PMLR}
}

@article{zhou2019sgd,
  title={{SGD} converges to global minimum in deep learning via star-convex path},
  author={Zhou, Yi and Yang, Junjie and Zhang, Huishuai and Liang, Yingbin and Tarokh, Vahid},
  journal={arXiv preprint arXiv:1901.00451},
  year={2019}
}

@article{hardt2018gradient,
  title={Gradient descent learns linear dynamical systems},
  author={Hardt, Moritz and Ma, Tengyu and Recht, Benjamin},
  journal={Journal of Machine Learning Research},
  volume={19},
  number={29},
  pages={1--44},
  year={2018}
}

@inproceedings{zhou2017stochastic,
    author = {Zhou, Zhengyuan and Mertikopoulos, Panayotis and Bambos, Nicholas and Boyd, Stephen and Glynn, Peter W},
    booktitle = {Advances in Neural Information Processing Systems},
    editor = {I. Guyon and U. Von Luxburg and S. Bengio and H. Wallach and R. Fergus and S. Vishwanathan and R. Garnett},
    pages = {},
    publisher = {Curran Associates, Inc.},
    title = {Stochastic Mirror Descent in Variationally Coherent Optimization Problems},
    volume = {30},
    year = {2017}
}

@article{craven1985invex,
    author = {Craven, B. and Glover, B.},
    year = {1985},
    month = {08},
    pages = {1 - 20},
    title = {Invex Functions and Duality},
    volume = {39},
    journal = {Journal of the Australian Mathematical Society}
}

@article{smith2017bayesian,
  title={A {B}ayesian perspective on generalization and stochastic gradient descent},
  author={Smith, Samuel L and Le, Quoc V},
  journal={arXiv preprint arXiv:1710.06451},
  year={2017}
}

@misc{jordan2024muon,
  author       = {Keller Jordan and Yuchen Jin and Vlado Boza and Jiacheng You and Franz Cesista and Laker Newhouse and Jeremy Bernstein},
  title        = {Muon: An optimizer for hidden layers in neural networks},
  year         = {2024}
}

@article{liu2023aiming,
  title={Aiming towards the minimizers: fast convergence of {SGD} for overparametrized problems},
  author={Liu, Chaoyue and Drusvyatskiy, Dmitriy and Belkin, Misha and Davis, Damek and Ma, Yian},
  journal={Advances in neural information processing systems},
  volume={36},
  pages={60748--60767},
  year={2023}
}

@incollection{mangasarian1975pseudo,
    title = {Pseudo-convex functions},
    editor = {W.T. ZIEMBA and R.G. VICKSON},
    booktitle = {Stochastic Optimization Models in Finance},
    publisher = {Academic Press},
    pages = {23-32},
    year = {1975},
    isbn = {978-0-12-780850-5},
    author = {Olvi L. Mangasarian}
}

@book{mangasarian1994nonlinear,
    title={Nonlinear Programming: Theory and Algorithms},
    publisher={SIAM},
    author={Olvi L. Mangasarian},
    place={Philadelphia},
    year={1994}
}

@article{bonnans1995second,
    author = {Joseph Frédéric Bonnans and Alexander Ioffe},
    journal = {Mathematics of Operations Research},
    number = {4},
    pages = {801--817},
    publisher = {INFORMS},
    title = {Second-Order Sufficiency and Quadratic Growth for Nonisolated Minima},
    urldate = {2026-01-18},
    volume = {20},
    year = {1995}
}

@article{fornasier2024consensus,
    author = {Fornasier, Massimo and Klock, Timo and Riedl, Konstantin},
    title = {Consensus-Based Optimization Methods Converge Globally},
    journal = {SIAM Journal on Optimization},
    volume = {34},
    number = {3},
    pages = {2973-3004},
    year = {2024},
    doi = {10.1137/22M1527805}
}

@article{fornasier2020consensus,
  title={Consensus-based optimization on hypersurfaces: Well-posedness and mean-field limit},
  author={Fornasier, Massimo and Huang, Hui and Pareschi, Lorenzo and S{\"u}nnen, Philippe},
  journal={Mathematical Models and Methods in Applied Sciences},
  volume={30},
  number={14},
  pages={2725--2751},
  year={2020},
  publisher={World Scientific}
}

@article{fornasier2025regularity,
  title={Regularity and positivity of solutions of the Consensus-Based Optimization equation: unconditional global convergence},
  author={Fornasier, Massimo and Sun, Lukang},
  journal={arXiv preprint arXiv:2502.01434},
  year={2025}
}

@article{kangal2018subspace,
    author = {Kangal, Fatih and Meerbergen, Karl and Mengi, Emre and Michiels, Wim},
    title = {A Subspace Method for Large-Scale Eigenvalue Optimization},
    journal = {SIAM Journal on Matrix Analysis and Applications},
    volume = {39},
    number = {1},
    pages = {48-82},
    year = {2018},
    doi = {10.1137/16M1070025}
}

@article{carmon2020acceleration,
  title={Acceleration with a ball optimization oracle},
  author={Carmon, Yair and Jambulapati, Arun and Jiang, Qijia and Jin, Yujia and Lee, Yin Tat and Sidford, Aaron and Tian, Kevin},
  journal={Advances in Neural Information Processing Systems},
  volume={33},
  pages={19052--19063},
  year={2020}
}

@inproceedings{carmon2021thinking,
  title={Thinking inside the ball: Near-optimal minimization of the maximal loss},
  author={Carmon, Yair and Jambulapati, Arun and Jin, Yujia and Sidford, Aaron},
  booktitle={Conference on Learning Theory},
  pages={866--882},
  year={2021},
  organization={PMLR}
}

@article{asi2021stochastic,
  title={Stochastic bias-reduced gradient methods},
  author={Asi, Hilal and Carmon, Yair and Jambulapati, Arun and Jin, Yujia and Sidford, Aaron},
  journal={Advances in Neural Information Processing Systems},
  volume={34},
  pages={10810--10822},
  year={2021}
}

@inproceedings{jin2017escape,
  title={How to escape saddle points efficiently},
  author={Jin, Chi and Ge, Rong and Netrapalli, Praneeth and Kakade, Sham M and Jordan, Michael I},
  booktitle={International conference on machine learning},
  pages={1724--1732},
  year={2017},
  organization={PMLR}
}

@inproceedings{karimi2016linear,
  title={Linear convergence of gradient and proximal-gradient methods under the polyak-{\l}ojasiewicz condition},
  author={Karimi, Hamed and Nutini, Julie and Schmidt, Mark},
  booktitle={Joint European conference on machine learning and knowledge discovery in databases},
  pages={795--811},
  year={2016},
  organization={Springer}
}

@article{khanh2025star,
  title={Star Quasiconvexity: an Unified Approach for Linear Convergence of First-Order Methods Beyond Convexity},
  author={Khanh, Phan Quoc and Lara, Felipe},
  journal={arXiv preprint arXiv:2510.24981},
  year={2025}
}

@article{polyak1963gradient,
    author = {Polyak, Boris},
    year = {1963},
    month = {12},
    pages = {864-878},
    title = {Gradient methods for the minimisation of functionals},
    volume = {3},
    journal = {USSR Computational Mathematics and Mathematical Physics}
}

@article{lojasiewicz1963topological,
    author = {{\L}ojasiewicz, Stanis{\l}aw},
    year = {1963},
    pages = {87–89},
    title = {A topological property of real analytic subsets (in French)},
    journal = {Coll. du CNRS, Les {\'e}quations, aux d{\'e}riv{\'e}es partielles}
}

@article{necoara2019linear,
  title={Linear convergence of first order methods for non-strongly convex optimization},
  author={Necoara, Ion and Nesterov, Yu and Glineur, Francois},
  journal={Mathematical programming},
  volume={175},
  number={1},
  pages={69--107},
  year={2019},
  publisher={Springer}
}

@article{jin2021nonconvex,
  title={On nonconvex optimization for machine learning: Gradients, stochasticity, and saddle points},
  author={Jin, Chi and Netrapalli, Praneeth and Ge, Rong and Kakade, Sham M and Jordan, Michael I},
  journal={Journal of the ACM (JACM)},
  volume={68},
  number={2},
  pages={1--29},
  year={2021},
  publisher={ACM New York, NY, USA}
}

@article{gruntkowska2025ball,
  title={The {B}all-{P}roximal (="{B}roximal") {P}oint {M}ethod: a New Algorithm, Convergence Theory, and Applications},
  author={Gruntkowska, Kaja and Li, Hanmin and Rane, Aadi and Richt{\'a}rik, Peter},
  journal={arXiv preprint arXiv:2502.02002},
  year={2025}
}

@article{kovalev2025understanding,
  title={Understanding Gradient Orthogonalization for Deep Learning via Non-{E}uclidean Trust-Region Optimization},
  author={Kovalev, Dmitry},
  journal={arXiv preprint arXiv:2503.12645},
  year={2025}
}

@inproceedings{ge2017no,
  title        = {No spurious local minima in nonconvex low rank problems: A unified geometric analysis},
  author       = {Ge, Rong and Jin, Chi and Zheng, Yi},
  booktitle    = {International Conference on Machine Learning},
  pages        = {1233--1242},
  year         = {2017},
  organization = {PMLR}
}

@article{defazio2025gradients,
  title={Why Gradients Rapidly Increase Near the End of Training},
  author={Defazio, Aaron},
  journal={arXiv preprint arXiv:2506.02285},
  year={2025}
}

@ARTICLE{koren2009matrix,
  author={Koren, Yehuda and Bell, Robert and Volinsky, Chris},
  journal={Computer}, 
  title={Matrix Factorization Techniques for Recommender Systems}, 
  year={2009},
  volume={42},
  number={8},
  pages={30-37},
  doi={10.1109/MC.2009.263}
}

@article{candes2015phase,
  title={Phase retrieval via Wirtinger flow: Theory and algorithms},
  author={Candes, Emmanuel J and Li, Xiaodong and Soltanolkotabi, Mahdi},
  journal={IEEE Transactions on Information Theory},
  volume={61},
  number={4},
  pages={1985--2007},
  year={2015},
  publisher={IEEE}
}

@inproceedings{meka2008rank,
    author = {Meka, Raghu and Jain, Prateek and Caramanis, Constantine and Dhillon, Inderjit S.},
    title = {Rank minimization via online learning},
    year = {2008},
    isbn = {9781605582054},
    publisher = {Association for Computing Machinery},
    address = {New York, NY, USA},
    doi = {10.1145/1390156.1390239},
    booktitle = {Proceedings of the 25th International Conference on Machine Learning},
    pages = {656–663},
    numpages = {8},
    location = {Helsinki, Finland},
    series = {ICML 2008}
}

@article{jain2017non,
  title={Non-convex optimization for machine learning},
  author={Jain, Prateek and Kar, Purushottam},
  journal={arXiv preprint arXiv:1712.07897},
  year={2017}
}

@article{murty1987someNP,
  title   = {Some {NP}-complete problems in quadratic and nonlinear programming},
  author  = {Katta G. Murty and Santosh N. Kabadi},
  journal = {Mathematical Programming},
  year    = {1987},
  volume  = {39},
  pages   = {117-129}
}

@book{nemirovski1983problem,
  title     = {Problem Complexity and Method Efficiency in Optimization},
  author    = {Nemirovski, A.S. and Yudin, D.B.},
  publisher = {Wiley},
  year      = {1983}
}

@article{nemirovskii1985optimal,
  title   = {Optimal methods of smooth convex minimization},
  journal = {USSR Computational Mathematics and Mathematical Physics},
  volume  = {25},
  number  = {2},
  pages   = {21-30},
  year    = {1985},
  issn    = {0041-5553},
  author  = {A.S. Nemirovski and Y.E. Nesterov}
}

@book{nesterov2003introductory,
  title     = {Introductory lectures on convex optimization: A basic course},
  author    = {Nesterov, Yurii},
  volume    = {87},
  year      = {2003},
  publisher = {Springer Science \& Business Media}
}

@article{pascanu2014saddle,
  title={On the saddle point problem for non-convex optimization},
  author={Pascanu, Razvan and Dauphin, Yann N and Ganguli, Surya and Bengio, Yoshua},
  journal={arXiv preprint arXiv:1405.4604},
  year={2014}
}

@article{patterson1934fourier,
  title = {A Fourier Series Method for the Determination of the Components of Interatomic Distances in Crystals},
  author = {Patterson, A. L.},
  journal = {Phys. Rev.},
  volume = {46},
  issue = {5},
  pages = {372--376},
  year = {1934},
  month = {Sep},
  publisher = {American Physical Society},
  doi = {10.1103/PhysRev.46.372}
}

@book{goodfellow2016deep,
    title={Deep Learning},
    author={Ian Goodfellow and Yoshua Bengio and Aaron Courville},
    publisher={MIT Press},
    year={2016}
}

@article{lecun2015deep,
  title = {Deep learning},
  author = {LeCun, Yann and Bengio, Yoshua and Hinton, Geoffrey},
  journal = {Nature},
  volume = {521},
  pages = {436--444},
  year = {2015},
  doi = {10.1038/nature14539}
}

@inproceedings{safran2018spurious,
  title={Spurious local minima are common in two-layer relu neural networks},
  author={Safran, Itay and Shamir, Ohad},
  booktitle={International conference on machine learning},
  pages={4433--4441},
  year={2018},
  organization={PMLR}
}

@article{vidal2017mathematics,
  title={Mathematics of deep learning},
  author={Vidal, Rene and Bruna, Joan and Giryes, Raja and Soatto, Stefano},
  journal={arXiv preprint arXiv:1712.04741},
  year={2017}
}

@article{ahn2023escape,
  title={How to escape sharp minima with random perturbations},
  author={Ahn, Kwangjun and Jadbabaie, Ali and Sra, Suvrit},
  journal={arXiv preprint arXiv:2305.15659},
  year={2023}
}

@article{rockafellar1976monotone,
  title     = {Monotone operators and the proximal point algorithm},
  author    = {Rockafellar, R Tyrrell},
  journal   = {SIAM Journal on Control and Optimization},
  volume    = {14},
  number    = {5},
  pages     = {877--898},
  year      = {1976},
  publisher = {SIAM}
}

@incollection{danilova2022recent,
  title     = {Recent theoretical advances in non-convex optimization},
  author    = {Danilova, Marina and Dvurechensky, Pavel and Gasnikov, Alexander and Gorbunov, Eduard and Guminov, Sergey and Kamzolov, Dmitry and Shibaev, Innokentiy},
  booktitle = {High-Dimensional Optimization and Probability: With a View Towards Data Science},
  pages     = {79--163},
  year      = {2022},
  publisher = {Springer}
}

@article{feizi2017porcupine,
  title   = {Porcupine neural networks: (almost) all local optima are global},
  author  = {Feizi, Soheil and Javadi, Hamid and Zhang, Jesse and Tse, David},
  journal = {arXiv preprint arXiv:1710.02196},
  year    = {2017}
}

@article{shin2022effects,
  title     = {Effects of depth, width, and initialization: A convergence analysis of layer-wise training for deep linear neural networks},
  author    = {Shin, Yeonjong},
  journal   = {Analysis and Applications},
  volume    = {20},
  number    = {01},
  pages     = {73--119},
  year      = {2022},
  publisher = {World Scientific}
}

@inproceedings{allen2019convergence,
  title        = {A convergence theory for deep learning via over-parameterization},
  author       = {Allen-Zhu, Zeyuan and Li, Yuanzhi and Song, Zhao},
  booktitle    = {International conference on machine learning},
  pages        = {242--252},
  year         = {2019},
  organization = {PMLR}
}

@INPROCEEDINGS{haeffele2017global,
  author={Haeffele, Benjamin D. and Vidal, René},
  booktitle={2017 IEEE Conference on Computer Vision and Pattern Recognition (CVPR)}, 
  title={Global Optimality in Neural Network Training}, 
  year={2017},
  volume={},
  number={},
  pages={4390-4398},
  doi={10.1109/CVPR.2017.467}}

@article{soltanolkotabi2018theoretical,
  title={Theoretical insights into the optimization landscape of over-parameterized shallow neural networks},
  author={Soltanolkotabi, Mahdi and Javanmard, Adel and Lee, Jason D},
  journal={IEEE Transactions on Information Theory},
  volume={65},
  number={2},
  pages={742--769},
  year={2018},
  publisher={IEEE}
}

@article{gruntkowska2025non,
  title={Non-Euclidean Broximal Point Method: A Blueprint for Geometry-Aware Optimization},
  author={Gruntkowska, Kaja and Richt{\'a}rik, Peter},
  journal={arXiv preprint arXiv:2510.00823},
  year={2025}
}

@article{arrow1961quasi,
 author = {Kenneth J. Arrow and Alain C. Enthoven},
 journal = {Econometrica},
 number = {4},
 pages = {779--800},
 publisher = {[Wiley, Econometric Society]},
 title = {Quasi-Concave Programming},
 urldate = {2026-01-14},
 volume = {29},
 year = {1961}
}

@book{laffont2002theory,
    ISBN = {9780691091846},
    author = {Jean-Jacques Laffont and David Martimort},
    publisher = {Princeton University Press},
    title = {The Theory of Incentives: The Principal-Agent Model},
    urldate = {2026-01-15},
    year = {2002}
}

@book{wolfstetter1999topics,
    place={Cambridge},
    title={Topics in Microeconomics: Industrial Organization, Auctions, and Incentives},
    publisher={Cambridge University Press},
    author={Wolfstetter, Elmar},
    year={1999}
}

@ARTICLE{ke2007quasiconvex,
  author={Ke, Qifa and Kanade, Takeo},
  journal={IEEE Transactions on Pattern Analysis and Machine Intelligence}, 
  title={Quasiconvex Optimization for Robust Geometric Reconstruction}, 
  year={2007},
  volume={29},
  number={10},
  pages={1834-1847},
  doi={10.1109/TPAMI.2007.1083}}

@INPROCEEDINGS{seiler2010quasiconvex,
  author={Seiler, Peter and Balas, Gary J.},
  booktitle={49th IEEE Conference on Decision and Control (CDC)}, 
  title={Quasiconvex sum-of-squares programming}, 
  year={2010},
  volume={},
  number={},
  pages={3337-3342},
  doi={10.1109/CDC.2010.5717672}}

@article{eppstein2005quasiconvex,
  title={Quasiconvex programming},
  author={Eppstein, David},
  journal={Combinatorial and Computational Geometry},
  volume={52},
  number={287-331},
  pages={3},
  year={2005},
  publisher={Cambridge University Press}
}

@article{ge2016matrix,
  title={Matrix completion has no spurious local minimum},
  author={Ge, Rong and Lee, Jason D and Ma, Tengyu},
  journal={Advances in neural information processing systems},
  volume={29},
  year={2016}
}

@book{bazaraa1993nonlinear,
    title={Nonlinear Programming: Theory and Algorithms},
    publisher={John Wiley \& Sons, Ltd},
    author={Bazaraa, Mokhtar S. and Sherali, Hanif D. and C. M. Shetty},
    place={New York},
    year={1993}
}
\bibliographystyle{icml2026}

\newpage

\appendix
\onecolumn

\section*{Appendix}

\section{Facts and Lemmas}

We first collect key definitions and present some fundamental results needed for the proofs in the remainder of the paper. The presentation and notation largely follows that of \citet{gruntkowska2025non}.

For any nonempty, closed, and convex set $\cX \subseteq \R^d$, we define its \emph{indicator function} $\delta_{\cX}: \R^d \to \R \cup \{+\infty\}$ by
\begin{align*}
    \delta_{\cX}(z) \eqdef
    \begin{cases}
        0 & \text{if } z\in \cX, \\
        +\infty & \text{otherwise}.
    \end{cases}
\end{align*}
This function is proper, closed, and convex.

\begin{definition}[Subdifferential]
    Let $f: \R^d \mapsto \R \cup \{+\infty\}$ be proper and let $x \in \dom(f)$. The \emph{subdifferential} of $f$ at $x$, denoted as $\partial f(x)$, is the set of vectors $g\in\R^d$ such that
    \begin{align*}
        f(y) \geq f(x) + \inp{g}{y-x} \qquad \forall y\in\R^d.
    \end{align*}
    The elements of $\partial f(x)$ are called the \emph{subgradients} of $f$ at $x$.
\end{definition}

\begin{fact}\label{fact:subdiff_id_ball}
    The subdifferential of the indicator function of a ball $\cB_{\mX}(x, t)$ is
    \begin{align*}
        \partial \delta_{\cB_{\mX}(x, t)}(y) = \cN_{\cB_{\mX}(x, t)}(y) = 
        \begin{cases}
            \R_{\geq0} \mX (y-x) & \norm{x-y}_{\mX} = t, \\
            \{0\} & \norm{x-y}_{\mX} < t, \\
            \emptyset & \norm{x-y}_{\mX} > t.
        \end{cases}
    \end{align*}
    where $\cN_{\cB_{\mX}(x, t)}(y)$ is the normal cone of $\cB_{\mX}(x, t)$ at $y$.
\end{fact}

\begin{lemma}\label{lemma:nonconv_diff_cone}
    Let $f: \R^d \mapsto \R$ be differentiable. Fix $x \in \R^d$ and let $u\in\argmin_{z\in \cB_{\mX}(x, t)} f(z)$. Then $- \nabla f(u) \in \cN_{\cB_{\mX}(x, t)}(u)$.
\end{lemma}
\begin{proof}
    By the argument in the proof of Lemma~3 in \citet{kovalev2025understanding}, we have $- \nabla f(u) \in \partial\delta_{\cB_{\mX}(x, t)}(u)$. The result then follows directly from \Cref{fact:subdiff_id_ball}.
\end{proof}

The following is an analogue of the First Brox Theorem \citep[Theorem D.1]{gruntkowska2025ball}. Note that unlike in the convex case considered by \citet{gruntkowska2025ball}, when convexity assumption is dropped, $\BProxSub{t}{f}{x}$ need not be a singleton lying on the boundary of $\cB_{\mX}(x, t)$ when $\cB_{\mX}(x, t) \cap \cX_f = \emptyset$ (see \Cref{fig:nonconv}).
\begin{theorem}\label{thm:1stbrox_nc}
    Let $f: \R^d \mapsto \R \cup \{+\infty\}$ be proper and closed and choose $x \in \dom f$. Then $\BProxSub{t}{f}{x} \neq \emptyset$. Moreover, if $\cB_{\mX}(x, t) \cap \cX_f \neq \emptyset$, then $\BProxSub{t}{f}{x}$ is a nonempty subset of $\cX_f$.
\end{theorem}
\begin{proof}
    The broximal operator is minimizing a proper, closed function over a closed set $\cB_{\mX}(x, t)$. Since such functions attain their minimum over compact sets, it follows that $\BProxSub{t}{f}{x} \neq \emptyset$. Moreover, if $\cB_{\mX}(x, t) \cap \cX_f \neq \emptyset$, then $\BProxSub{t}{f}{x} \subseteq \cX_f$ is nonempty.
\end{proof}

The next lemma shows that under \Cref{as:main1}, \algnamesmall{BPM} cannot get stuck away from the global minimum and keep ``jumping'' between two points.
\begin{lemma}\label{lemma:2ptdecr}
    Let \Cref{as:main1} hold and suppose that $u_1 \in \BProxSub{t}{f}{u_2}$ and $u_2 \in \BProxSub{t}{f}{u_1}$. If $u_1, u_2 \not\in \cX_f$, then $u_1=u_2$.
\end{lemma}

\begin{proof}
    Suppose that $u_1 \in \BProxSub{t}{f}{u_2}$, $u_2 \in \BProxSub{t}{f}{u_1}$ and $u_1, u_2 \not\in \cX_f$. Since $u_1 \in \argmin_{z \in \cB_{\mX}(u_2, t)} f(z)$ and $u_1 \notin \cX_f$, we have $\dist{u_2}{\cX_f} > t$, and likewise $\dist{u_1}{\cX_f} > t$. Therefore, $u_1, u_2 \notin \cX_f + \cB_{\mX}(0,t)$. Then, by \Cref{as:main1}, we have
    \begin{align*}
        \inp{u_2-u_1}{u_1-x_{\star}}_{\mX} \geq 0, \\
        \inp{u_1-u_2}{u_2-x_{\star}}_{\mX} \geq 0.
    \end{align*}
    Adding the two inequalities gives
    \begin{align*}
        0 \leq \inp{u_1-u_2}{u_2 - x_{\star} - u_1 + x_{\star}}_{\mX} = - \norm{u_1-u_2}_{\mX}^2,
    \end{align*}
    and hence $u_1=u_2$.
\end{proof}

The result below is a direct consequence of \Cref{as:main2}.

\begin{lemma}\label{lemma:iters_neq}
    Let \Cref{as:main2} hold and let $\{x_k\}_{k\geq0}$ be the iterates of \algnamesmall{BPM}. If $x_k \not\in \cX_f$, then $x_k \neq x_{k+1}$.
\end{lemma}
\begin{proof}
    Since $x_k \not\in \cX_f$, \Cref{as:main2} says that $\BProxSub{t}{f}{x_k} \neq \{x_k\}$, and hence the set $\BProxSub{t}{f}{x_k}$ has cardinality greater than $1$. By definition of the algorithm, it follows that $x_k \neq x_{k+1}$.
\end{proof}

\begin{remark}
    Note that by the reasoning above, under Assumptions \ref{as:main1} and \ref{as:main2}, we know that the algorithm has already reached $\cX_f$ when the function value does not decrease for two consecutive iterations.
\end{remark}

\begin{lemma}\label{lemma:t_dist}
    Let Assumptions \ref{as:main1} and \ref{as:main2} hold. Then, for any iteration $k$ such that $x_{k+1}\not\in\cX_f$, the iterates of \algnamesmall{BPM} satisfy either
    \begin{itemize}
        \item $\norm{x_k - x_{k+2}}_{\mX} > t$, or 
        \item $\norm{x_k - x_{k+3}}_{\mX} > t$ and $\norm{x_{k+1} - x_{k+3}}_{\mX} > t$.
    \end{itemize}
    Therefore, $\norm{x_k - x_{k+3}}_{\mX} > t$ for any such $k$.
\end{lemma}
\begin{proof}
    Since $x_{k+1}\not\in\cX_f$, we have $x_k \not\in \cX_f$, so from \Cref{lemma:iters_neq} we know that $x_k \neq x_{k+1}$ and $x_{k+1} \neq x_{k+2}$. By design, for any $j$ we have $f(x_j) \geq f(x_{j+1})$. Now, let us consider two cases:

    \textbf{Case 1: $f(x_k) = f(x_{k+1})$.} By \Cref{lemma:2ptdecr}, $x_k \not\in \BProxSub{t}{f}{x_{k+1}}$. Since $\BProxSub{t}{f}{x_{k+1}}$ is non-empty (\Cref{thm:1stbrox_nc}), there exists $u \in \BProxSub{t}{f}{x_{k+1}}$ such that $f(u)<f(x_{k+1})$ (the inequality is strict because otherwise we would have $x_k \in \BProxSub{t}{f}{x_{k+1}}$). Consequently, $u \not\in B(x_k,t)$ for any $u \in \BProxSub{t}{f}{x_{k+1}}$, meaning that $\norm{x_k-x_{k+2}}_{\mX} > t$. Since $f(x_{k+3}) \leq f(x_{k+2}) < f(x_k)$, clearly $\norm{x_k - x_{k+3}}_{\mX} > t$.
    
    \textbf{Case 2: $f(x_k) > f(x_{k+1})$.} If $\norm{x_k-x_{k+2}}_{\mX} > t$, then the claim holds. Now, suppose that $\norm{x_k-x_{k+2}}_{\mX} \leq t$. Then $x_{k+1}, x_{k+2} \in \BProxSub{t}{f}{x_k}$, which forces $f(x_{k+1}) = f(x_{k+2})$. As in Case~1, a plateau in function values must be followed by a strict decrease, so $f(x_{k+3}) < f(x_{k+2})$. This in turn means that $x_{k+3} \not\in \BProxSub{t}{f}{x_k}$ and $x_{k+3} \not\in \BProxSub{t}{f}{x_{k+1}}$, which completes the proof.
\end{proof}

\newpage

\section{Proof of \Cref{thm:nc_bpm}}

\MAINTHM*

\begin{proof}
    \begin{enumerate}[label=(\roman*)]
        \item That $x_{k+1} \in \cX_f$ if $\cX_f\cap \cB_{\mX}(x_k, t)\neq\emptyset$ is obvious from the definition of broximal operator. Similarly, if $\cX_f\cap \cB_{\mX}(x_k, t)=\emptyset$, then $x_{k+1}$ cannot belong to $\cX_f$.
        
        \item We start with the simple decomposition
        \begin{eqnarray*}
            \norm{x_{k+1} - x_{\star}}_{\mX}^2 = \norm{x_k-x_{\star}}_{\mX}^2 - 2 \inp{x_k - x_{k+1}}{x_{k+1} - x_{\star}}_{\mX} - \norm{x_{k+1} - x_k}_{\mX}^2.
        \end{eqnarray*}
        Since $x_k \not\in \cX_f + \cB_{\mX}(0,t)$, applying \Cref{as:main1}, we get
        \begin{eqnarray}\label{eq:nvwoiaq}
            \norm{x_{k+1} - x_{\star}}_{\mX}^2 &\leq& \norm{x_k-x_{\star}}_{\mX}^2 - \norm{x_{k+1} - x_k}_{\mX}^2 \\
            &<& \norm{x_k-x_{\star}}_{\mX}^2, \nonumber
        \end{eqnarray}
        so the distances $\brac{\norm{x_k-x_{\star}}_{\mX}}_{k\geq0}$ are decreasing.
        Now, from \Cref{lemma:t_dist} and triangle inequality, it follows that
        \begin{eqnarray*}
            t < \norm{x_{k-2} - x_{k+1}}_{\mX} \leq \norm{x_{k-2} - x_{k-1}}_{\mX} + \norm{x_{k-1} - x_k}_{\mX} + \norm{x_k - x_{k+1}}_{\mX},
        \end{eqnarray*}
        and the AM-QM inequality gives
        \begin{eqnarray*}
            &&\hspace{-1.5cm}\sqrt{\frac{\norm{x_{k-2} - x_{k-1}}^2_{\mX} + \norm{x_{k-1} - x_k}^2_{\mX} + \norm{x_k - x_{k+1}}^2_{\mX}}{3}} \\
            &\geq& \frac{\norm{x_{k-2} - x_{k-1}}_{\mX} + \norm{x_{k-1} - x_k}_{\mX} + \norm{x_k - x_{k+1}}_{\mX}}{3}
            > \frac{t}{3}.
        \end{eqnarray*}
        This in turn implies that $\norm{x_{k-2} - x_{k-1}}^2_{\mX} + \norm{x_{k-1} - x_k}^2_{\mX} + \norm{x_k - x_{k+1}}^2_{\mX} > \frac{t^2}{3}$.
        Hence, applying \eqref{eq:nvwoiaq}, we get
        \begin{eqnarray}\label{eq:rahgfdav}
            \norm{x_{k+1} - x_{\star}}_{\mX}^2 &\leq& \norm{x_{k-1}-x_{\star}}_{\mX}^2 - \parens{\norm{x_k - x_{k-1}}_{\mX}^2 + \norm{x_{k+1} - x_k}_{\mX}^2} \nonumber \\
            &\leq& \norm{x_{k-2}-x_{\star}}_{\mX}^2 - \parens{\norm{x_{k-1} - x_{k-2}}_{\mX}^2 + \norm{x_k - x_{k-1}}_{\mX}^2 + \norm{x_{k+1} - x_k}_{\mX}^2} \nonumber \\
            &<& \norm{x_{k-2}-x_{\star}}_{\mX}^2 - \frac{t^2}{3}.
        \end{eqnarray}
        
        \item Applying \eqref{eq:rahgfdav} iteratively, we get
        \begin{eqnarray*}
            \norm{x_K - x_{\star}}_{\mX}^2 &\leq& \norm{x_{K-3}-x_{\star}}_{\mX}^2 - \frac{t^2}{3} \\
            &\leq& \norm{x_{(K \textnormal{mod} \, 3)}-x_{\star}}_{\mX}^2 - \flr{\frac{K}{3}} \frac{t^2}{3} \\
            &\overset{\eqref{eq:nvwoiaq}}{<}& \norm{x_0-x_{\star}}_{\mX}^2 - \flr{\frac{K}{3}} \frac{t^2}{3},
        \end{eqnarray*}
        and the conclusion follows.
    \end{enumerate}
\end{proof}

\newpage

\section{Proofs for \Cref{sec:examples}}

In this section, we prove the results from \Cref{sec:examples} identifying the classes of potentially non-convex functions from the optimization literature under which Assumptions \ref{as:main1} and \ref{as:main2} hold. We begin by elaborating on the examples from \Cref{sec:examples} (\Cref{sec:ap_examples}), then move on to presenting transformations of the loss under which our assumptions are preserved (\Cref{sec:transform}), and finish off with providing proofs of the implications presented in \Cref{sec:examples}, connecting Assumptions \ref{as:main1} and \ref{as:main2} to existing literature (Sections \ref{sec:quasiconv}, \ref{sec:quasarconv} and \ref{sec:aiming}).

Unless otherwise stated, we specialize to $\mX = \mI$; the exception is \Cref{sec:transform}, where the general $\mX$ setting is used.

\subsection{Examples}\label{sec:ap_examples}

We first provide the proofs for the results in \Cref{sec:toy_examples}. 

\EXAMPLE*

Indeed, fix $t>0$ and take any $x \in \dom f \backslash \brac{\cX_f + \cB(0,t)}$. Then $f(z) = \norm{z}^2$ for all $z\in \cB(x, t)$, and the broximal operator in this case will return $u = x - t\frac{x}{\norm{x}}$. Therefore, choosing $x_\star = 0$,
\begin{align*}
    \inp{x - u}{u - x_\star} = \inp{t \frac{x}{\norm{x}}}{x - t\frac{x}{\norm{x}}} = t(\norm{x} - t) > 0,
\end{align*}
so \Cref{as:main1} holds.

Lastly, if $\BProxSub{t}{f}{x} = \brac{x}$, then $x$ must be a local minimizer of $f$ within the ball $\cB(x, t)$.
The only points that could satisfy this are the global minimizers $0$ and $a$, and hence 
\begin{align*}
 \BProxSub{t}{f}{x} = \brac{x} \implies x \in \brac{0, a} = \cX_f. 
\end{align*}

The example above motivates the construction of a larger class of functions that satisfy Assumptions \ref{as:main1} and \ref{as:main2}.

\begin{lemma}\label{lemma:constructing}
    Let $g: \R^d \mapsto \R \cup \brac{+\infty}$ be a proper, closed function with a nonempty set of minimizers $\cX_g$. Suppose further that $g$ satisfies Assumptions~\ref{as:main1} and~\ref{as:main2} with some radius $t > 0$. Let $\cC \subseteq \dom g$ be any closed set such that
    \begin{align*}
        \cC\cap \cX_g = \emptyset
    \end{align*}
    and define $f: \R^d \mapsto \R \cup \brac{+\infty}$ by 
    \begin{align*}
        f(x) =
        \begin{cases}
            g_{\inf} & x \in \cC, \\
            g(x) & \text{otherwise},
         \end{cases}
    \end{align*}
    where $g_{\inf}$ is the minimum of $g$. Then the following statements hold:
    \begin{enumerate}
        \item $f$ is proper and closed, and its set of global minimizers is $\cX_f = \cX_g \cup \cC$.
        \item $f$ satisfies Assumption~\ref{as:main1} with the same radius $t$ and the same choice of $x_\star$.
        \item $f$ satisfies Assumption~\ref{as:main2} with the same radius $t$.
    \end{enumerate}
\end{lemma}

\begin{proof}
    We prove each claim in turn.
    \begin{enumerate}
        \item Since $g$ is proper, we have $\dom g \neq \emptyset$ and $g > -\infty$ everywhere. By construction, $f(x) \geq g_{\inf} > -\infty$ for all $x$, and
        \begin{align*}
            \dom f = \dom g \neq \emptyset.
        \end{align*}
        Hence $f$ is proper. We next verify that $f$ is closed, i.e., lower semicontinuous.

        \textbf{Case 1: $x \notin \cC$.}
        Since $\cC$ is closed, there exists $\varepsilon > 0$ such that $\cB(x, \varepsilon) \cap \cC = \emptyset$.
        Therefore, for any sequence $\{x_k\}$ with $x_k \to x$, we have $x_k \notin \cC$ for all sufficiently large $k$, and hence $f(x_k) = g(x_k)$.
        Since $g$ is lower semicontinuous at $x$, it follows that $f$ is lower-semicontinuous at $x$ as well.

        \textbf{Case 2: $x \in \cC$.}
        Then $f(x) = g_{\inf}$. By construction, $f \geq g \geq g_{\inf}$ everywhere, so for any sequence $x_k \to x$,
        \begin{align*}
        \liminf_{k \to \infty} f(x_k) \geq g_{\inf} = f(x).
        \end{align*}
        
        Combining the two cases, $f$ is lower semicontinuous and hence closed.
        Lastly, $\cX_f = \cX_g \cup \cC$ by construction.
        
        \item Let $x \in \dom f$ satisfy $\dist{x}{\cX_f} > t$, and let $u \in \BProxSub{t}{f}{x}$.
        Then the ball $\cB(x, t)$ does not intersect $\cX_f$, and in particular does not intersect $\cC$.
        Consequently, $f = g$ on $\cB(x, t)$, which implies
        \begin{align*}
        \BProxSub{t}{f}{x} = \BProxSub{t}{g}{x}.
        \end{align*}
        
        Since $\cX_g \subseteq \cX_f$, we also have $\dist{x}{\cX_g} > t$, so Assumption~\ref{as:main1} for $g$ applies.
        Therefore, there exists $x_\star \in \cX_g$ such that
        \begin{align*}
        \inp{x - u}{u - x_\star} \geq 0.
        \end{align*}
        This establishes Assumption~\ref{as:main1} for $f$ with the same $x_\star$.
        
        \item Assume that $\BProxSub{t}{f}{x} = \{x\}$ for some $x \in \dom f$. If $x \in \cC$, then $x \in \cX_f$ and the claim holds. Now, suppose $x \notin \cC$. If $\cB(x, t) \cap \cC \neq \emptyset$, choose $z \in \cB(x, t) \cap \cC$. Then
        \begin{align*}
        f(z) = g_{\inf} \leq f(x).
        \end{align*}
        Since $x$ is the unique minimizer of $f$ over $\cB(x, t)$, this forces $z = x$, contradicting $x \notin \cC$. Hence $\cB(x, t) \cap \cC = \emptyset$. Therefore, $f = g$ on $\cB(x, t)$ and
        \begin{align*}
        \BProxSub{t}{f}{x} = \BProxSub{t}{g}{x} = \{x\}.
        \end{align*}
        By Assumption~\ref{as:main2} for $g$, we conclude that $x \in \cX_g \subseteq \cX_f$. Thus Assumption~\ref{as:main2} also holds for $f$.
    \end{enumerate}
\end{proof}

\subsection{Operations that Preserve Assumptions~\ref{as:main1} and~\ref{as:main2}}\label{sec:transform}

The next three lemmas apply in the general $\mX$ setting.

We first show that strictly increasing transformations preserve Assumptions~\ref{as:main1} and~\ref{as:main2}.

\MONOTRANS*

\begin{proof}
    We first note that $\dom f = \dom h$. Since $h$ is proper, $f$ is also proper.

    Now, we claim that $\cX_h = \cX_f$. Let $x_\star \in \cX_h$. Then for all $x \in \dom h = \dom f$, we have $h(x_\star) \leq h(x)$. Since $g$ is strictly increasing, this implies
    \begin{align*}
    f(x_\star) = g(h(x_\star)) \leq g(h(x)) = f(x),
    \end{align*}
    and hence $x_\star \in \cX_f$.
    Conversely, suppose $x_\star \in \cX_f$. Then for all $x \in \dom f$,
    \begin{align*}
    g(h(x_\star)) \leq g(h(x)).
    \end{align*}
    Strict monotonicity of $g$ implies $h(x_\star) \leq h(x)$, and thus $x_\star \in \cX_h$. Therefore, $\cX_f = \cX_h$.

    Now, fix any $x \in \dom h$ and define
    \begin{align*}
    h_x \eqdef h + \delta_{\cB_{\mX}(x, t)}, \qquad f_x \eqdef g \circ h_x = f + \delta_{\cB_{\mX}(x, t)}.
    \end{align*}
    Applying the argument above to $h_x$ and $f_x$ and using the definition of the broximal mapping yields
    \begin{align*}
    \BProxSub{t}{h}{x} = \cX_{h_x} = \cX_{f_x} = \BProxSub{t}{f}{x}.
    \end{align*}
    Since $\BProxSub{t}{f}{x} = \BProxSub{t}{h}{x}$ for all $x$, and since $\cX_f = \cX_h$, Assumption~\ref{as:main1} for $f$ follows immediately from Assumption~\ref{as:main1} for $h$ with the same choice of $x_\star$.

    Lastly, if $\BProxSub{t}{f}{x} = \{x\}$, then the equality above implies
    \begin{align*}
        \BProxSub{t}{h}{x} = \{x\}.
    \end{align*}
    By Assumption~\ref{as:main2} for $h$, we conclude that $x \in \cX_h = \cX_f$. Hence Assumption~\ref{as:main2} holds for $f$, so \Cref{as:main2} holds for~$f$ as well.
\end{proof}

We next show that Assumptions~\ref{as:main1} and~\ref{as:main2} are invariant under Euclidean isometries.

\EUCISO*

\begin{proof}
    Since $h$ is proper and $\phi$ is bijective, $f$ is proper.
    Moreover,
    \begin{align*}
        x \in \cX_f
        \quad\iff\quad
        h(\phi(x)) = \inf h
        \quad\iff\quad
        \phi(x) \in \cX_h,
    \end{align*}
    and hence $\cX_f = \phi^{-1}(\cX_h)\neq\emptyset$.
    
    Now, by the definition of $\mQ$, for any $x,y\in\R^d$,
    \begin{align*}
        \norm{\phi(y)-\phi(x)}_{\mX}^2
        = (y-x)^\top \mQ^\top \mX \mQ (y-x)
        = \norm{y-x}_{\mX}^2,
    \end{align*}
    so $y\in \cB_{\mX}(x, t)$ if and only if $\phi(y)\in \cB_{\mX}(\phi(x), t)$. Using the properties of the broximal operator, one can easily check that for any $x\in\dom f$
    \begin{align*}
    \BProxSub{t}{f}{x} = \phi^{-1}\parens{\BProxSub{t}{h}{\phi(x)}}.
    \end{align*}
    
    Next, we verify that \Cref{as:main1} holds. Let $x\in \dom f \setminus \brac{\cX_f + \cB_{\mX}(0,t)}$ and $u\in\BProxSub{t}{f}{x}$.
    Then $\phi(x)\in \dom h \setminus  \brac{\cX_h + \cB_{\mX}(0,t)}$ and $\phi(u)\in\BProxSub{t}{h}{\phi(x)}$.
    By Assumption~\ref{as:main1} for $h$, there exists $x_\star\in\cX_h$ such that
    \begin{align*}
        \inp{\phi(x)-\phi(u)}{\phi(u)-x_\star}_{\mX} \geq 0.
    \end{align*}
    Using $\mQ^\top \mX \mQ = \mX$, this is equivalent to
    \begin{align*}
        \inp{x-u}{u-\phi^{-1}(x_\star)}_{\mX} \geq 0.
    \end{align*}
    Since $\phi^{-1}(x_\star)\in\cX_f$, Assumption~\ref{as:main1} holds for $f$.
    
    Lastly, if $\BProxSub{t}{f}{x}=\{x\}$, then $\BProxSub{t}{h}{\phi(x)}=\{\phi(x)\}$. By Assumption~\ref{as:main2} for $h$, $\phi(x)\in\cX_h$, and hence $x\in\cX_f$.
\end{proof}

The following lemma is, in fact, a direct corollary of \Cref{lemma:transformation-11}. For completeness, we include a full proof.

\AFFFUNC*

\begin{proof}
    Since $g$ is proper, $f$ is also proper. Moreover, the set of minimizers is unchanged under positive scaling and additive shifts, so $\cX_f = \cX_g \neq \emptyset$.
    
    Note that for any $x \in \dom f$ and radius $t>0$, by definition of the broximal operator:
    \begin{align*}
        u \in \BProxSub{t}{f}{x} \quad &\iff \quad u \in \argmin_{z \in \cB_{\mX}(x,t)} f(z) = \argmin_{z \in \cB_{\mX}(x,t)} (a g(z) + b) \\
        &\iff \quad u \in \argmin_{z \in \cB_{\mX}(x,t)} g(z) \\
        &\iff \quad u \in \BProxSub{t}{g}{x}.
    \end{align*}
    Hence, we have $\BProxSub{t}{f}{x} = \BProxSub{t}{g}{x}$. Now, let $x \in \dom f \setminus \brac{\cX_f + \cB_{\mX}(0,t)} = \dom g \setminus \brac{\cX_g + \cB_{\mX}(0,t)}$ and $u \in \BProxSub{t}{f}{x} = \BProxSub{t}{g}{x}$. It is obvious that \Cref{as:main1} holds from $f$ whenever it holds for $g$.
    
    Lastly, if $\BProxSub{t}{f}{x} = \{x\}$, then
    \begin{align*}
    \BProxSub{t}{g}{x} = \BProxSub{t}{f}{x} = \{x\}.
    \end{align*}
    By Assumption~\ref{as:main2} for $g$, we have $x \in \cX_g = \cX_f$, so \Cref{as:main2} also holds for $f$ with radius $t$.
\end{proof}

\subsection{Quasiconvexity}\label{sec:quasiconv}

\QUASICONVDEF*

The proof of the main result in this section will use the characterization of quasiconvexity via sublevel sets.

\begin{lemma}\label{lemma:quasi_conv_lf_conv}
    Let $f$ be quasiconvex. Then, for any $x\in\R^d$, the sublevel sets $L_f^{\leq}(x) \eqdef \brac{u \in \R^d: f(u) \leq f(x)}$ and $L_f^{<}(x) \eqdef \brac{u \in \R^d: f(u) < f(x)}$ are convex.
\end{lemma}
\begin{remark}
    In fact, one can show that this is an ``if and only if'' statement, i.e., $L_f^{\leq}(x)$ being convex for all $x$ implies quasiconvexity.
\end{remark}
\begin{proof}
    Let $u_1, u_2 \in L_f^{\leq}(x)$. Then, for any $\lambda \in [0,1]$, quasiconvexity gives
    \begin{align*}
        f\parens{(1-\lambda)u_1 + \lambda u_2} \leq \max\brac{f(u_1), f(u_2)} \leq f(x),
    \end{align*}
    proving the first result. The proof of the second statement is entirely analogous.
\end{proof}

\QUASITHM*

\begin{proof}
    Fix some $x \in \dom f \backslash \brac{\cX_f + \cB(0,t)}$ and suppose for a contradiction that there exists $u\in\BProxSub{t}{f}{x}$ such that
    \begin{align*}
        \inp{x-u}{u-x_{\star}} < 0.
    \end{align*}

    Let $p \eqdef \Proj{[u, x_\star]}{x}$ be the projection of $x$ onto the line passing through $u$ and $x_\star$. Since we assume that $\inp{x-u}{u-x_{\star}} < 0$, $p$ must lie strictly between $u$ and $x_\star$. But then
    \begin{align*}
        \norm{x - p}^2 = \norm{x - u}^2 - \norm{u - p}^2
        < \norm{x - u}^2,
    \end{align*}
    and hence $p \in \cB(x, t)$, so by definition of broximal operator $f(p) \geq f(u)$. On the other hand, strict quasiconvexity implies
    \begin{align*}
        f(p) < \max\brac{f(u), f_\star} = f(u).
    \end{align*}
    This contradiction shows that \Cref{as:main1} holds.
        
    Now, suppose that $\BProxSub{t}{f}{x} = \{x\}$. Then $x$ is the unique minimizer of $f$ on $\cB(x, t)$, and so by quasiconvexity it must be the global minimizer (otherwise $L_f^{\leq}(x)$ would not be connected, and so would not be convex).
\end{proof}

\subsection{Pseudoconvexity}\label{sec:pseudoconv}

\PSEUDOCONVDEF*

\PSEUDOTHM*

\begin{proof}
    Fix some $x \in \dom f \backslash \brac{\cX_f + \cB(0,t)}$ and suppose for a contradiction that there exists $u\in\BProxSub{t}{f}{x}$ such that
    \begin{align*}
        \inp{x-u}{u-x_{\star}} < 0.
    \end{align*}
    By \Cref{lemma:nonconv_diff_cone}, we have
    \begin{align*}
        - \nabla f(u) \in \cN_{\cB(x, t)}(u) \overset{\eqref{fact:subdiff_id_ball}}{=}
        \begin{cases}
            \R_{\geq0} (u-x) & \norm{x-u} = t, \\
            \{0\} & \norm{x-u} < t.
        \end{cases}
    \end{align*}
    Since pseudoconvexity implies that every stationary point is a global minimizer, and because $x \in \dom f \backslash \brac{\cX_f + \cB(0,t)}$, we must have $u \not\in \cX_f$. Therefore, there exists $c\geq0$ such that $\nabla f(u) = c(x-u)$. Moreover, $c>0$, since otherwise $\nabla f(u)=0$ and the same contradiction would arise. Using this representation, we obtain
    \begin{align*}
        0 > \inp{x-u}{u-x_{\star}}
        = \frac{1}{c} \inp{\nabla f(u)}{u-x_{\star}}.
    \end{align*}
    By pseudoconvexity, this implies $f(x_{\star}) \geq f(u)$, contradicting the fact that $u \not\in \cX_f$. This contradiction establishes \Cref{as:main1}.

    \Cref{as:main2} holds since under pseudoconvexity any stationary point is a global minimizer.
\end{proof}

\subsubsection{Generalized Pseudoconvexity}\label{sec:gen_pseudoconv}

\Cref{def:pseudoconvex} is stated for differentiable functions. We now extend the comparison to the nondifferentiable setting by considering generalized notions of pseudoconvexity, and show that our assumptions remain strictly weaker in this broader regime.

To this end, we work with an abstract notion of a generalized subdifferential. Specifically, we only require the mapping $\partial f$ to satisfy the following properties:
\begin{itemize}
    \item $0 \in \partial f(x)$ whenever $x \in \dom(f)$ is a local minimizer of $f$,
    \item if $u \in \argmin \limits_{z \in \cB_{\mX}(x, t)} f(z)$, then $0 \in \partial f(u) + \cN_{\cB_{\mX}(x, t)}(u)$.
\end{itemize}

\begin{definition}[Generalized Pseudoconvexity]\label{as:gen-pseu-convx}
    Let $f: \R^d \mapsto \R\cup \brac{+\infty}$ be a proper function and let $\partial f \subseteq \R^d$ denote a generalized subdifferential defined for all all $x \in \dom f$.
    We say that $f$ is \emph{generalized pseudoconvex} if 
    \begin{align*}
     \parens{\exists g \in \partial f(x) \,\,\text{s.t.}\, \inp{g}{y - x} \geq 0} \quad\implies\quad f(y) \geq f(x)
    \end{align*} 
    for all $x, y \in \dom f$.
\end{definition}

\begin{remark}
    In the differentiable case, where $\partial f(x) = \brac{\nabla f(x)}$, the above condition reduces to 
    \begin{align*}
        \inp{\nabla f(x)}{y - x} \geq 0 \quad\implies\quad f(y) \geq f(x),
    \end{align*}
    which is the standard definition of pseudoconvexity.
\end{remark}

\begin{lemma}
    Suppose that $f$ is generalized pseudoconvex. Then $f \in \newclass{\mI}{t}$ for any $t>0$.
\end{lemma}

\begin{proof}
    Fix any $x \in \dom f \backslash \brac{\cX_f + \cB(0,t)}$ and let $u = \BProxSub{t}{f}{x}$. Taking the contrapositive of the condition in \Cref{as:gen-pseu-convx}, we have
    \begin{align*}
        f(x) > f(y) 
        \implies 
        \parens{\inp{g}{y - x} < 0 \quad \forall g \in \partial f(x)},
    \end{align*}
    and hence, choosing $x=u$ and $y=x_\star \in \cX_f$ yields
    \begin{align*}
        \inp{g}{x_\star - u} < 0 \quad \forall g \in \partial f(u).
    \end{align*}
    By the assumed properties of the generalized subdifferential, there exists $g_u \in \partial f(u)$ such that $- g_u \in \cN_{\cB_{\mX}(x, t)}(u)$. Since under generalized pseudoconvexity every local minimizer is global and $u\not\in\cX_f$, there exists $c \geq 0$ such that $g_u = c(x - u)$ (\Cref{fact:subdiff_id_ball}). Substituting this expression above gives 
    \begin{align*}
        \inp{c(x - u)}{x_\star - u} < 0.
    \end{align*}
    Thus $c > 0$ and $\inp{x - u}{u - x_\star} > 0$, which establishes \Cref{as:main1} for any choice of $x_\star \in \cX_f$.
    
    Finally, \Cref{as:main2} follows again from the fact that generalized pseudoconvexity ensures that all local minimizers are global minimizers.
\end{proof}

\subsection{Quasar convexity}\label{sec:quasarconv}

\QUASARCONVDEF*

\begin{lemma}\label{eq:quasar_equiv}
    Let $f: \R^d \to \R$ be $\zeta$-quasar convex with respect to $x_{\star}\in\cX_f$. Then
    \begin{align*}
        f\parens{\lambda x_\star + (1-\lambda) x} \leq \zeta \lambda f_\star + (1-\zeta \lambda) f(x)
    \end{align*}
    for all $x\in\R^d$ and $\lambda\in[0,1]$.
\end{lemma}
\begin{proof}
    The result follows from Lemma 10 by \citet{hinder2019near}.
\end{proof}

\begin{lemma}\label{lemma:quasar_proj}
    Suppose that $f$ is $\zeta$-quasar convex with respect to $x_{\star}\in\cX_f$ and take any $x\in\R^d \backslash \brac{\cX_f + \cB(0,t)}$, where $t>0$. Let $u=\BProxSub{t}{f}{x}$ and let $\Pi(x)$ be the projection of $x$ onto the line passing through $u$ and $x_{\star}$. Then there exists $\lambda\in[0,1)$ such that $u = \lambda x_{\star} + (1-\lambda) \Pi(x)$.
\end{lemma}
\begin{proof}
    First, note that $\norm{x - \Pi(x)} \leq \norm{x-u} \leq t$, so $\Pi(x) \in \cB(x, t)$, meaning that $f(u) \leq f(\Pi(x))$ by the definition of broximal operator. Clearly, $u \neq x_\star$ and $\Pi(x) \neq x_\star$ since $x\in\R^d \backslash \brac{\cX_f + \cB(0,t)}$. If $u=\Pi(x)$, then the claim of the lemma holds with $\lambda=0$. Now, suppose that the three points $u$, $x_{\star}$ and $\Pi(x)$ are distinct. There are $3$ possibilities: $\Pi(x)$ lies between $u$ and $x_{\star}$, $x_{\star}$ lies between $u$ and $\Pi(x)$, or $u$ lies between $\Pi(x)$ and $x_{\star}$.
    That $x_{\star}$ does not lie between $u$ and $\Pi(x)$ is obvious, since if that was the case, then $x_{\star}$ would belong to the ball $\cB(0,t)$, contradicting the fact that $x\in\R^d \backslash \brac{\cX_f + \cB(0,t)}$.
    Assume that $\Pi(x)$ lies between $u$ and $x_{\star}$, i.e., there exists $\lambda\in(0,1)$ such that $\Pi(x) = \lambda x_{\star} + (1-\lambda) u$. Then, \Cref{eq:quasar_equiv} gives
    \begin{align*}
        f(\Pi(x))
        = f(\lambda x_{\star} + (1-\lambda) u)
        \leq \zeta \lambda f_{\star} + (1 - \zeta \lambda) f(u)
        < \zeta \lambda f(\Pi(x)) + (1 - \zeta \lambda) f(\Pi(x)),
    \end{align*}
    where the last inequality follows from the fact that $f_{\star} < f(u) \leq f(\Pi(x))$.
    This contradiction finishes the proof.
\end{proof}

\QUASARTHM*

\begin{proof}
    By \Cref{lemma:quasar_proj}, we have
    \begin{eqnarray*}
        \inp{x-u}{x_{\star}-u}
        &=& (1-\lambda) \inp{x - u}{x_{\star} - \Pi(x)} \\
        &=& (1-\lambda) \inp{x - \Pi(x)}{x_{\star} - \Pi(x)} + (1-\lambda) \inp{\Pi(x) - u}{x_{\star} - \Pi(x)} \\
        &=& 0 - (1-\lambda) \norm{\Pi(x) - u} \norm{\Pi(x) - x_{\star}}
        \leq 0,
    \end{eqnarray*}
    so \Cref{as:main1} holds.

    Now, suppose that $\BProxSub{t}{f}{x} = \{x\}$. Then $x$ is the unique minimizer of $f$ on $\cB_{\mX}(x, t)$, and so it must be the global minimizer (otherwise $\cX_f$ would not be connected, which cannot hold under quasar convexity -- see \Cref{eq:quasar_equiv}).
\end{proof}

\subsection{Aiming Condition}\label{sec:aiming}

\AIMINGDEF*

\begin{lemma}\label{lemma:aiming_part}
    Let $f$ satisfy the aiming condition. Then \Cref{as:main2} holds with any $t>0$ and
    \begin{align*}
        \inp{x-u}{u-\bar{u}} \geq 0
    \end{align*}
    for any $x \in \dom f \backslash \brac{\cX_f + \cB(0,t)}$, where $\bar{u} \in \Proj{\cX_f}{u}$.
\end{lemma}
\begin{proof}
    The proof follows the same structure as that of \Cref{thm:pseudo_conv_as}. For completeness, we present the full argument again.
    Let $x \in \dom f \backslash \brac{\cX_f + \cB(0,t)}$ and $u \in \BProxSub{t}{f}{x}$.
    By \Cref{lemma:nonconv_diff_cone}, we have
    \begin{align*}
        - \nabla f(u) \in \cN_{\cB(x, t)}(u) \overset{\eqref{fact:subdiff_id_ball}}{=}
        \begin{cases}
            \R_{\geq0} (u-x) & \norm{x-u} = t, \\
            \{0\} & \norm{x-u} < t.
        \end{cases}
    \end{align*}
    Since the aiming condition implies that every stationary point is a global minimizer, and because $x \in \dom f \backslash \brac{\cX_f + \cB(0,t)}$, we must have $u \not\in \cX_f$. Therefore, there exists $c\geq0$ such that $\nabla f(u) = c(x-u)$. In fact, $c$ must be strictly positive. Indeed, if we had $c=0$, then $u$ would be a global minimizer, again contradicting the choice of $x$. Now, let $\bar{u} \in \Proj{\cX_f}{u}$ be the point for which \eqref{eq:aiming} holds. Then
    \begin{align*}
        \inp{x-u}{u-\bar{u}} = \frac{1}{c} \inp{\nabla f(u)}{u-\bar{u}}
        \geq \frac{\theta}{c} f(u) \geq 0,
    \end{align*}
    proving the claim.
\end{proof}

\AIMINGTHM*
\begin{proof}
    By \Cref{lemma:aiming_part}, for any $x \in \dom f \backslash \brac{\cX_f + \cB(0,t)}$ and $t>0$, we have
    \begin{align*}
        \inp{x-u}{u-\bar{u}} \geq 0,
    \end{align*}
    where $\bar{u} \in \Proj{\cX_f}{u}$. Since the minimizer is unique, $\cX_f = \{x_\star\}$ is a singleton, and hence \Cref{as:main1} holds.

    Now, suppose that $\BProxSub{t}{f}{x} = \{x\}$. Then $x$ is the unique minimizer of $f$ on $\cB(x, t)$, and so $\theta f(x) \leq \inp{\nabla f(x)}{x - \bar{x}} = 0$, meaning that $x\in\cX_f$. Hence \Cref{as:main2} holds as well.
\end{proof}

Lastly, we establish the connectedness of$\cX_f$ under aiming condition under the aiming condition (as used in \Cref{tab:summary}).

\begin{proposition}[Connectedness of minimizers under aiming]\label{prop:aiming_connected}
    Let $f:\R^d\to\R$ be differentiable and satisfy the aiming condition. Then the set of global minimizers $\cX_f$ is connected.
\end{proposition}

\begin{proof}
    Since $\cX_f = f^{-1}(\{0\})$ and $f$ is continuous, $\cX_f$ is closed. Assume for contradiction that $\cX_f$ is disconnected. Then there exist two nonempty, disjoint, closed, connected components $C_1,C_2\subseteq\cX_f$ such that $\cX_f = C_1 \cup C_2$ and $C_1\cap C_2=\emptyset$. Among all such pairs, choose $C_1$, $C_2$ so that the distance between them is minimal. Since $C_1$ and $C_2$ are disjoint closed sets in $\R^d$, the distance
    \begin{align*}
        d \eqdef \inf_{x\in C_1,\, y\in C_2} \norm{x-y}
    \end{align*}
    is strictly positive. Moreover, by closedness, there exist points $x_1\in C_1$ and $x_2\in C_2$ such that $\norm{x_1-x_2}=d$. Consider the line segment
    \begin{eqnarray*}
        \gamma(t) \eqdef (1-t)x_1 + t x_2,\qquad t\in[0,1],
    \end{eqnarray*}
    and define the function $\varphi(t) \eqdef f(\gamma(t))$. Since $f$ is differentiable, the function $\varphi$ is differentiable on $[0,1]$. Moreover,
    \begin{eqnarray*}
        \varphi(0)=f(x_1)=0,\qquad \varphi(1)=f(x_2)=0,
    \end{eqnarray*}
    so by Rolle's theorem, there exists $t_0\in(0,1)$ such that $\varphi'(t_0)=0$. Using the chain rule, we obtain
    \begin{eqnarray*}
        \varphi'(t) = \inp{\nabla f(\gamma(t))}{x_2-x_1},
    \end{eqnarray*}
    and hence
    \begin{eqnarray}\label{eq:oaihegrn}
        \inp{\nabla f(\gamma(t_0))}{x_2-x_1} = 0.
    \end{eqnarray}
    Now, let $\bar{x} \in \Proj{\cX_f}{\gamma(t_0)}$ be a projection point for which the aiming condition holds. We claim that $\bar{x}$ is either $x_1$ or $x_2$. Indeed, suppose there exists another point $z \in \Proj{\cX_f}{\gamma(t_0)}$ such that $z\not\in\{x_1, x_2\}$. By the minimality of $d$ and the triangle inequality, we have
    \begin{eqnarray*}
        \norm{x_1 - \gamma(t_0)} + \norm{\gamma(t_0) - x_2} = d \leq \norm{x_1 - z} \leq \norm{x_1 - \gamma(t_0)} + \norm{\gamma(t_0) - z}.
    \end{eqnarray*}
    It follows that $\norm{\gamma(t_0) - x_2} \leq \norm{\gamma(t_0) - z}$. On the other hand, since $z \in \Proj{\cX_f}{\gamma(t_0)}$, we also have $\norm{\gamma(t_0) - x_2} \geq \norm{\gamma(t_0) - z}$, meaning that $\norm{\gamma(t_0) - x_2} = \norm{\gamma(t_0) - z}$. But then
    \begin{eqnarray*}
        \norm{x_1 - z}^2 &=& \norm{x_1 - \gamma(t_0)}^2 + \norm{z - \gamma(t_0)}^2 - 2 \inp{x_1 - \gamma(t_0)}{z - \gamma(t_0)} \\
        &=& \norm{x_1 - \gamma(t_0)}^2 + \norm{z - \gamma(t_0)}^2 - 2 \norm{x_1 - \gamma(t_0)} \norm{z - \gamma(t_0)} \cos(\theta_1) \\
        &<& \norm{x_1 - \gamma(t_0)}^2 + \norm{z - \gamma(t_0)}^2 + 2 \norm{x_1 - \gamma(t_0)} \norm{z - \gamma(t_0)} \\
        &=& \norm{x_1 - \gamma(t_0)}^2 + \norm{x_2 - \gamma(t_0)}^2 - 2 \norm{x_1 - \gamma(t_0)} \norm{x_2 - \gamma(t_0)} \cos(\theta_2) \\
        &=& \norm{x_1 - \gamma(t_0)}^2 + \norm{x_2 - \gamma(t_0)}^2 - 2 \inp{x_1 - \gamma(t_0)}{x_2 - \gamma(t_0)} \\
        &=& \norm{x_1 - x_2}^2
    \end{eqnarray*}
    where $\theta_1$ is the angle between the vectors $x_1 - \gamma(t_0)$ and $z - \gamma(t_0)$ and $\theta_2$ is the angle between the vectors $x_1 - \gamma(t_0)$ and $x_2 - \gamma(t_0)$. This contradicts the minimality of $d$.

    So suppose without loss of generality that $\bar{x} = x_1$. Then the vector $\gamma(t_0) - \bar{x} = \gamma(t_0) - x_1$ is parallel to $x_2-x_1$, and hence, by aiming condition and \eqref{eq:oaihegrn},
    \begin{eqnarray*}
        \theta f(\gamma(t_0)) \leq \inp{\nabla f(\gamma(t_0))}{\gamma(t_0) - \bar{x}} = 0.
    \end{eqnarray*}
    Since $\theta>0$, it follows that $f(\gamma(t_0))=0$, i.e., $\gamma(t_0)\in\cX_f$. However, since $t_0\in(0,1)$, we have $\norm{x_1 - \gamma(t_0)} < d$, which contradicts the minimality of $d$. Therefore, $\cX_f$ must be connected.
\end{proof}

\newpage

\section{Monotonicity of Function Classes}

In this section, we study how the classes of functions satisfying our main assumptions behave as the radius parameter~$t$ changes.  We show that functions satisfying only Assumption~\ref{as:main2} form a monotone class, while functions satisfying Assumption~\ref{as:main1} do not in general exhibit monotonicity in $t$.

That said, if monotonicity in the radius is desired, there exists a subclass of $\newclass{\mX}{t}$, which we term the \emph{Uniform Broximal Alignment} class (denoted $\newclassu{\mX}{t}$; see \Cref{sec:uba}), that \emph{is} monotone in $t$.

We first define the class of functions that satisfy Assumption~\ref{as:main2} alone.

\begin{definition}\label{def:assump2}
    For $t>0$, define
    \begin{align*}
        \cF_2(t) \eqdef& \brac{f : f \text{ is proper and satisfies \Cref{as:main2} with } t > 0} \\
        =& \brac{f \text{ proper} : \forall x \in \dom f, \,\BProxSub{t}{f}{x} = \{x\} \,\implies\, x \in \cX_f}.
    \end{align*}
\end{definition}

The following proposition shows that $\cF_2(t)$ is monotone with respect to $t$.

\begin{proposition}\label{prop:monotonicityt-asmp2}
    For $0 < t_1 \leq t_2$, we have
    \begin{align*}
        \cF_2(t_1) \subseteq \cF_2(t_2).
    \end{align*}
\end{proposition}

\begin{proof}
    Take any $f \in \cF_2(t_1)$ and fix $x \in \dom f$ such that $\BProxSub{t_2}{f}{x} = \{x\}$. We want to show that $x \in \cX_f$.
    Since $t_1 \leq t_2$, we have $B_{t_1}(x) \subseteq B_{t_2}(x)$, so $x$ is also the unique minimizer of $f$ over $B_{t_1}(x)$, i.e.,
    \begin{align*}
        \BProxSub{t_1}{f}{x} = \{x\}.
    \end{align*}
    By definition of $\cF_2(t_1)$, we know that $x \in \cX_f$, and hence $f \in \cF_2(t_2)$.
\end{proof}
Thus, the class $\cF_2(t)$ grows as $t$ increases.

Next, we consider the class of functions satisfying Assumption~\ref{as:main1}:
\begin{definition}\label{def:gf1t}
    For $t>0$, define
    \begin{align*}
        \cF_1(t) \eqdef \brac{f: f \text{ is proper and satisfies \Cref{as:main1} with radius } t > 0}.
    \end{align*}
\end{definition}

Unlike $\cF_2(t)$, the class $\cF_1(t)$ is generally \emph{not} monotone in $t$.

\begin{proposition}\label{prop:nonmonotone}
The class $\cF_1(t)$ is in general not monotone in $t$: for $t_1 < t_2$, neither $\cF_1(t_1) \subseteq \cF_1(t_2)$ nor $\cF_1(t_2) \subseteq \cF_1(t_1)$ holds in general.
\end{proposition}

\begin{proof}
    We provide two concrete examples.

    \textbf{Example 1:} $f \in \cF_1(1)$ but $f \notin \cF_1(2)$.

    Let $f:\R^2\mapsto\R$ be defined on 
    \begin{align*}
        \dom f = \{(0,0), (1,0), (2,0), (3,0), (3,2)\},
    \end{align*}
    with values
    \begin{align*}
        f(0, 0) = 0,\quad f(1, 0) = 1,\quad f(2, 0) = 2,\quad f(3, 0) = 3,\quad f(3, 2) = 0.5.
    \end{align*}
    One can easily check that for radius $t=1$, all points outside $\dom f \backslash \brac{\cX_f + B(0,1)} = \{(2,0), (3,0), (3,2)\}$ satisfy Assumption~\ref{as:main1}:
    \begin{itemize}
        \item $x=(2,0)$: feasible points $B_1(x) \cap \dom f = \{(1,0),(2,0),(3,0)\}$, so $u = \BProxSub{1}{f}{x} = \{(1,0)\}$, and
        \begin{align*}
            \inp{x-u}{u-x_\star} = \inp{(2,0)-(1,0)}{(1,0)-(0,0)} = 1 \geq 0.
        \end{align*}
        
        \item $x=(3,0)$: feasible points $\{(2,0),(3,0)\}$, so $u = \BProxSub{1}{f}{x} = \{(2,0)\}$, and
        \begin{align*}
            \inp{x-u}{u-x_\star} = \inp{(3,0)-(2,0)}{(2,0)-(0,0)} = 2 \geq 0.
        \end{align*}
        
        \item $x=(3,2)$: feasible points $\{(3,2)\}$, so $u = \BProxSub{1}{f}{x} = \{(3,2)\}$, and
        \begin{align*}
            \inp{x-u}{u-x_\star} = 0.
        \end{align*}
    \end{itemize}
    Hence $f \in \cF_1(1)$. Now, we show that $f \notin \cF_1(2)$. Indeed, consider $x=(3,0)$. Then $\BProxSub{2}{f}{x} = \{(3,2)\}$ and
    \begin{align*}
        \inp{x-u}{u-x_\star} = \inp{(3,0)-(3,2)}{(3,2)-(0,0)} = -4 < 0,
    \end{align*}
    which violates Assumption~\ref{as:main1}. Hence $f \notin \cF_1(2)$.

    \textbf{Example 2:} $f \in \cF_1(3)$ but $f \notin \cF_1(1)$.

    Let $f:\R^2\mapsto\R$ be defined on
    \begin{align*}
        \dom f = \{(0,0),(2,0),(2,1)\},
    \end{align*}
    with values
    \begin{align*}
        f(0,0)=0,\quad f(2,0)=1,\quad f(2,1)=0.1,
    \end{align*}
    with global minimizer $x_\star = (0,0)$.  
    
    We first show that $f \notin \cF_1(1)$. Consider $x=(2,0)$. The feasible points in $B_1(x)$ are $(2,0)$ and $(2,1)$, and $u = \BProxSub{1}{f}{x} = \{(2,1)\}$.
    Then
    \begin{align*}
        \inp{x-u}{u-x_\star} = \inp{(2,0)-(2,1)}{(2,1)-(0,0)} = -1 < 0,
    \end{align*}
    violating Assumption~\ref{as:main1}. Hence $f \notin \cF_1(1)$.
    
    Now, we show that $f \in \cF_1(3)$. For $t=3$, observe that
   \begin{align*}
        \norm{(2,0)} = 2 < 3,
        \qquad
        \norm{(2,1)} = \sqrt{5} < 3,
   \end{align*}
   meaning that
    \begin{align*}
        \dom f \subset \cX_f + B_3(0).
    \end{align*}
    Therefore, there are no points outside $\cX_f + B_3(0)$, and Assumption~\ref{as:main1} is trivially satisfied. Hence $f \in \cF_1(3)$.
\end{proof}

\subsection{Uniform Broximal Alignment}\label{sec:uba}

\begin{definition}[Uniform Broximal Alignment]\label{as:uba}
    Let $f:\R^d \to \R \cup \{+\infty\}$ be a proper, closed function. We say that $f$ satisfies the \emph{uniform broximal alignment with radius $t > 0$} if there exists $x_{\star} \in \cX_f$ such that, for every $$x \in \dom f \backslash \brac{\cX_f + \cB_{\mX}(0,t)}$$
    and every $z \in \dom f$ satisfying $$f(z) \leq f(x),$$
    we have
    \begin{align*}
        \inp{x-z}{x_{\star}-z}_{\mX} \leq 0.
    \end{align*}
    If this holds, we will write $f\in \newclassu{\mX}{t}$.
\end{definition}

Uniform alignment is strictly stronger than broximal alignment: it requires the alignment inequality for all points $z$ with $f(z) \leq f(x)$, whereas \Cref{as:main1} from our paper only requires it for those $z$ that are broximal minimizers of $f$ over $\cB_{\mX}(x, t)$.

We now show that the uniform alignment condition is monotonic in the radius.

\begin{theorem}
    Let $f:\R^d \to \R \cup \{+\infty\}$ be a proper, closed function and let $t' \geq t > 0$. Then
    \begin{align*}
        \newclassu{\mX}{t} \subseteq \newclassu{\mX}{t'}.
    \end{align*}
\end{theorem}

\begin{proof}
    Let $f \in \newclassu{\mX}{t}$ and fix any $t' \geq t$. We will show that $f$ satisfies the uniform broximal alignment condition with radius $t'$. Since $f\in \newclassu{\mX}{t}$, there exists $x_{\star} \in \cX_f$ such that, for every $x \in \dom f \backslash \brac{\cX_f + \cB_{\mX}(0,t)}$ and every $z \in \dom f$ satisfying $f(z) \leq f(x)$, we have
    \begin{align}\label{eq:aojberf}
        \inp{x-z}{x_{\star}-z}_{\mX} \leq 0.
    \end{align}
    Now, let $x \in \dom f \backslash \brac{\cX_f + \cB_{\mX}(0,t')}$ and let $z \in \dom f$ satisfy $f(z) \leq f(x)$. Because $t' \geq t$, we have $\cB_{\mX}(0,t) \subseteq \cB_{\mX}(0,t)$, and hence
    \begin{align*}
        \cX_f + \cB_{\mX}(0,t) \subseteq \cX_f + \cB_{\mX}(0,t').
    \end{align*}
    Therefore,
    \begin{align*}
        x \not\in \cX_f + \cB_{\mX}(0,t') \qquad \implies \qquad x \not\in \cX_f + \cB_{\mX}(0,t).
    \end{align*}
    So the pair $(x, z)$ satisfies the assumptions of \eqref{eq:aojberf}, and thus
    \begin{align*}
        \inp{x-z}{x_{\star}-z}_{\mX} \leq 0.
    \end{align*}
    Since this holds for every $x \in \dom f \backslash \brac{\cX_f + \cB_{\mX}(0,t')}$ and every $z \in \dom f$ with $f(z) \leq f(x)$, we conclude that $f \in \newclassu{\mX}{t'}$. The proof is complete.
\end{proof}

\newpage

\section{Notation}
\label{app:notation}

\bgroup
\def\arraystretch{1.3}
\begin{table}[H]
	\small
	\centering
	\begin{tabular}{|c|p{12cm}|}
	\hline
	\multicolumn{2}{|c|}{Notation} \\
	\hline
        $d$ & Dimensionality of the problem \\
	    $x_k$ & $k$-th iterate of an algorithm \\
        $\nabla f(x)$ & Gradient of function $f$ at $x$ \\
        $\partial f(x)$ & Subdifferential of function $f$ at $x$ \\
        $\cX_f$ & Set of global minimizers of $f$ \\
        $f_\star$ & Minimum of $f$ \\
        $\dom(f)$ & $\eqdef \{x \in \R^d : f(x)<+\infty\}$ -- domain of $f$ \\
        $\inp{\cdot}{\cdot}$ & Standard Euclidean inner product \\
        $\norm{\cdot}$ & Standard Euclidean norm \\
        $\inp{\cdot}{\cdot}_{\mX}$ & $\inp{x}{y}_{\mX} \eqdef \inp{\mX x}{y}$ for a symmetric positive definite matrix $\mX \in \R^{d\times d}$ \\
        $\norm{\cdot}_{\mX}$ & $\norm{x}_{\mX} \eqdef \sqrt{x^\top \mX x}$ for a symmetric positive definite matrix $\mX \in \R^{d\times d}$ \\
        $\cB(x, t)$ & $\eqdef \brac{z \in \R^d : \norm{z - x} \leq t}$ \\
        $\cB_{\mX}(x, t)$ & $\eqdef \brac{z \in \R^d : \norm{z - x}_{\mX} \leq t}$ \\
        $\BProxSub{t}{f}{x}$ & Broximal operator associated with function $f$ with radius $t>0$ \\
        $\newclass{\mX}{t}$ & The class of Broximal Aligned functions (see \Cref{def:BA}) \\
        $L_f^{<}(x)$ & $\eqdef \brac{u \in \R^d: f(u) < f(x)}$ \\
        $L_f^{\leq}(x)$ & $\eqdef \brac{z \in \R^d: f(z) \leq f(x)}$ \\
        $\Pi(\cdot, \cX)$ & Projection onto a set $\cX$ \\
        $\delta_{\cX}(y)$ & $=\begin{cases} 0, & y\in \cX \\ +\infty, & y\notin \cX \end{cases}$ \\
        $\textnormal{dist}(x, \cX)$ & $\eqdef \inf_{z\in\cX} \norm{x-z}_{\mX}$ \\
        $\cN_{\cX}(x)$ & $ \eqdef \brac{g\in\R^d : \inp{g}{z - x} \leq 0 \,\forall z \in \cX}$ -- the normal cone of $\cX$ at $x$ \\
       $ \R_{\geq0}(z)$ & $\eqdef \{\lambda z \;:\; \lambda \geq 0\}$ \\
        \hline
	\end{tabular}
	\caption{Frequently used notation.}
	\label{table:notation}
\end{table}
\egroup

\end{document}